\documentclass[a4paper,11pt]{amsart}

\usepackage{amsmath,amsfonts,amssymb,graphics,amsthm,mathrsfs,pgfplots}
\pgfplotsset{compat=1.18}

\usepackage{hyperref}
\usepackage{enumitem}
\usepackage{bbm}
\usepackage{mathtools}
\usepackage{a4wide}
\usepackage{enumitem}
\usepackage{dsfont}
\usepackage[foot]{amsaddr}

\newtheorem{theorem}{Theorem}[section]
\newtheorem{lemma}[theorem]{Lemma}
\newtheorem{proposition}[theorem]{Proposition}
\newtheorem{corollary}[theorem]{Corollary}

\newtheorem{conjecture}[theorem]{Conjecture}
\newtheorem{definition}[theorem]{Definition}
\theoremstyle{remark}
\newtheorem{remark}[theorem]{Remark}
\newtheorem{remarks}[theorem]{Remarks}

\newtheorem{claim}[theorem]{Claim}
\newtheorem{conditions}[theorem]{Conditions}

\newcommand*{\rowfor}[4]{#1 &#2 &#3 &#4}
\newcommand{\mtrfor}[4]{\begin{pmatrix} #1 \\#2 \\ #3 \\ #4 \end{pmatrix}}

\newcommand{\id}{\mathds{1}}
\newcommand{\R}{\mathbb{R}}
\renewcommand{\P}{\mathbb{P}}
\newcommand{\E}{\mathbb{E}}
\newcommand{\Z}{\mathbb{Z}}
\newcommand{\calP}{\mathcal{P}}
\newcommand{\calZ}{\mathcal{Z}}
\newcommand{\calG}{\mathcal{G}}
\newcommand{\calL}{\mathcal{L}}
\newcommand{\eps}{\varepsilon}

\begin{document}

\title[Largest gaps in Gaussian zeros]{Poisson approximation of the largest gaps between zeros of a stationary Gaussian process} 
\author{Renjie Feng}
\email{renjie.feng@sydney.edu.au}
\address{Sydney Mathematical Research Institute, University of Sydney}
\author{Stephen Muirhead}
\email{stephen.muirhead@monash.edu}
\address{School of Mathematics, Monash University}

\date{} 
\date{\today}

\begin{abstract}
We study the largest gaps between successive zeros of a smooth stationary Gaussian process. Our main result is that, if correlations decay at least polynomially, then after suitable rescaling of the locations and sizes of the largest gaps in a growing interval, the resulting joint process converges to a Poisson point process. The main novel step in the proof is to establish an approximate splitting property, with multiplicative error, for gap events in well-separated intervals; notably we achieve this for processes with arbitrarily slow polynomial decay of correlations.
\end{abstract}

\maketitle


\section{Introduction}

Let $f$ be a smooth stationary non-zero centred Gaussian process on $\R$ (SGP), and let $\mathcal{Z} = \{ x \in \R : f(x) = 0\}$ denote its countable zero set. We are interested in the largest gaps between successive zeros. It is natural to suppose that, under suitable assumptions on the process, the collection of largest gaps obeys a \textit{Poisson approximation}: restricted to a large interval, the locations of the largest gaps are approximately uniformly distributed and their sizes are approximately independent.

\smallskip
If the SGP is replaced by an i.i.d.\ sequence of Gaussian variables, and the `zeros' are interpreted as the indices at which the sequence changes sign, then the `gaps' are distributed as runs in a sequence of coin tosses. In that context the validity of the Poisson approximation is a classical result \cite{gsw86}. 

\smallskip Much less is known in the case of dependent processes \cite{nov92}. For Markov processes, one can leverage the fact that the gaps form an i.i.d.\ sequence, thus the analysis of large gaps falls into the setting of classical extreme value theory \cite{nov88}. However SGPs are far from being Markovian in general. In a different direction, the Poisson approximation has been established for the rescaled largest gaps between eigenvalues of the Gaussian unitary ensemble by exploiting its determinantal structure \cite{fw25}. Later the fluctuations of the largest gaps between eigenvalues were shown to be universal for general Wigner matrices \cite{bou22,llm20}. Related results for other random matrix ensembles include the study of the largest gaps in the complex Ginibre ensemble \cite{lo25} and in unitary-invariant Hermitian ensembles with general confining potentials \cite{cc26}. The latter extends the results of \cite{fw25} beyond the Gaussian case; in both works, the analysis relies on the determinantal structure of the eigenvalues.

\smallskip

 In this paper we give general conditions under which the largest gaps between successive zeros of a smooth stationary Gaussian process obey the Poisson approximation. Clearly not every stationary Gaussian process satisfies the Poisson approximation -- consider the constant Gaussian process, which has no gaps. Our main result verifies the Poisson approximation for a wide class of processes whose correlations decay at least polynomially, i.e.\ that satisfy $K(x) x^\alpha \to c \geq 0$ as $x \to \infty$ for some $\alpha > 0$, where $K(x) = \E[f(0)f(x)]$ is the covariance kernel. Notably we allow for arbitrarily small $\alpha$, and so cover the regime $\alpha \in (0,1)$ of long-range correlations in which the scaling of the largest gaps is anomalous (see Corollary~\ref{c:large} for a precise statement).
 
\smallskip 
Much more is known about the analogous question for the \textit{smallest} gaps between successive zeros, which is a simpler `local' statistic. Recently Poisson approximation has been shown to hold for the smallest gaps between zeros of SGPs with \textit{arbitrary} decay of correlations \cite{fgy24}. There has also been extensive work on smallest gaps between zeros of the characteristic polynomial of random matrices \cite{ftw, FW, FLY,c25}, as well as those of random analytic functions and random polynomials \cite{FY, MY}.  

\subsection{Main results} 
To state our results we first make the Poisson approximation precise, beginning with a suitable scaling function. For an interval $I \subset \R$, let  $\mathcal{G}_I$ be the gap (or `hole') event that $I \cap \mathcal{Z}$ is empty, and define 
\begin{equation}
    \label{e:gdef}
G(r) :=  \P[\mathcal{G}_{[0,r]}] .
\end{equation}
Clearly $G$ is non-increasing, and as we show later (see Lemma \ref{l:lambda}), $G$ is convex and twice differentiable under general conditions. The \textit{scaling function} is \begin{equation}\label{thetafunction}\theta(r) := -\log \lambda(r),\,\,\text{where} \,\, \lambda(r) = - G'(r).\end{equation} 
This can be interpreted as the \textit{log-intensity} of gaps of size $\ge r$. Indeed we show (see Lemma~\ref{l:lambda}) that
\[  \lambda(r) =  \E \big[  \big| \{  z \in [0,1] \cap \mathcal{Z} : f(z + \cdot) \in \mathcal{G}_{(0,r)}  \} \big| \big]  ,  \]
which by the Kac-Rice formula can also be expressed as 
\[ \lambda(r) =  (2 \pi K(0) )^{-1/2} \E \big[  |f'(0)|  \id_{ \mathcal{G}_{(0,r)} }  | f(0) = 0 \big] . \]

\smallskip
Next we define the point process of rescaled gaps.  Let $(z_i)_{i \in \Z} = \calZ$ be an arbitrary increasing ordering of the zeros. For $R > 1$ define the point process 
\begin{equation}
\label{e:psi}  \Psi_R := \sum_{ i : z_i \in [0,R]} \delta_{ ( z_i / R,   \,\, \theta (z_{i+1}-z_i) -  \log R ) }  
\end{equation}
on $[0,1] \times \R$, where $\delta$ denotes a Dirac delta mass; this encodes the location and size of the (rescaled) gaps in the interval $[0,R]$ (by convention the `location' of a gap is its left endpoint). Define the \textit{largest gap} in the interval $[0,R]$ and its \textit{location} to be respectively 
 \begin{equation}
 \label{e:def}
  \mathcal{L}_R :=  \max \{ z_{i+1} - z_i :  0 \le z_i \le R \}  \quad \text{and} \quad \calZ_R:= \{ 0 \le z_i \le R: z_{i+1} - z_i = \calL_R \}. 
  \end{equation}
For concreteness, if $\mathcal{Z} \cap [0,R]$ is empty then we set $\calL_R = \calZ_R = 0$, and if $\calZ_R$ is multiply defined (which may occur for degenerate processes) then we fix $\calZ_R$ to take its smallest value. Let $\R \cup \infty$ denote the compactification of $\R$ at $\infty$.

\begin{definition}
\label{d:pa}
We say that the `largest gaps obey a Poisson approximation' if $\Psi_R$ converges vaguely, as $R\to\infty$, to a Poisson point process on $[0,1] \times (\R \cup \infty)$ with intensity $dx \otimes e^{-y} dy$. 
\end{definition}

\noindent In particular this implies joint convergence in law of the location and size of the largest gap: \begin{equation}
\label{e:law}
\Big(\frac{\calZ_R}{R} ,  \theta(\calL_R)  - \log R \Big) \Longrightarrow   (  \mathcal{U} ,\mathcal{X} ) \, , \quad  \text{as } R \to \infty,
\end{equation}
where $(\mathcal{U},\mathcal{X})$ denotes a pair of independent standard uniform and Gumbel random variables. 

\smallskip
As mentioned, we verify the Poisson approximation for a wide class of SGPs whose correlations decay at least polynomially. We say that $f$ satisfies  \textit{polynomial decay with exponent $\alpha > 0$} if 
\begin{equation*}
\text{($\alpha$-PD) \qquad  $K(x) \sim x^{-\alpha}$ as $x \to \infty$,}
\end{equation*}
and \textit{super polynomial decay} if 
\begin{equation*}
\text{($\infty$-PD) \qquad for every $\alpha > 0$, $K(x) x^\alpha \to 0$ as $x \to \infty$.}
\end{equation*}

The \textit{spectral measure} of $f$ is the positive measure $\mu$ such that $K = \mathcal{F}[\mu]$, where $\mathcal{F}$ denotes the Fourier transform. We say that $f$ \textit{has a spectral density} if $\mu$ has a density.  

\smallskip
Our main result is the following:

\begin{theorem}
\label{t:main}
Let $f$ be a $C^2$-smooth SGP with a spectral density. Suppose either:
\begin{enumerate}
\item ($\infty$-PD) holds, and moreover $\int K(x) \,dx \neq 0$, or
\item ($\alpha$-PD) holds for some $\alpha \neq 1$, and moreover $K > 0$ and $\sup_x |K'(x)| / K(x) < \infty$.
\end{enumerate}
Then the largest gaps obey a Poisson approximation in the sense of Definition \ref{d:pa}.
\end{theorem}

We omit the case $\alpha = 1$ for technical reasons, although we expect that the Poisson approximation also holds in that case. In fact we believe that the Poisson approximation holds under much milder conditions (see Conjecture~\ref{c}).

\smallskip
To extract from Theorem \ref{t:main} quantitative information about the largest gap $\mathcal{L}_R$ requires some knowledge of the function $\theta(r) = - \log (- G'(r))$. Recently there has been much progress in determining the asymptotics of the related quantity $-\log G(r)$ \cite{ffm25,ffm26}, which turn out to depend heavily on the rate of decay of $K$.  By comparing $G$ and its derivative, we deduce from \cite{ffm25,ffm26} the following asymptotics:

\begin{proposition}
\label{p:per2}
Let $f$ be a $C^2$-smooth SGP with a spectral density.
\begin{enumerate}
\item If either ($\infty$-PD) or ($\alpha$-PD) is satisfied with $\alpha > 1$, and if also $\int K(x) \,dx \neq 0$, then there exists $\zeta > 0$ such that
\[ \theta(r) \sim \zeta r  \ , \quad \text{as } r \to \infty . \]
\item If ($\alpha$-PD) is satisfied with $\alpha \in (0,1)$, then there exists $\zeta > 0$ such that
\[ \theta(r) \sim \zeta r^\alpha (\log r)  \ , \quad \text{as } r \to \infty . \]
 In this case $\zeta$ can be written explicitly as
 \[ \zeta = \frac{K(0) \sqrt{\pi}  (1-\alpha)}{\Gamma((1-\alpha)/2) \Gamma(1+\alpha/2) } =  \frac{2 K(0)  (1-\alpha) \sin(\alpha \pi /2) }{ \sqrt{\pi} \alpha} . \]
 \end{enumerate}
 \end{proposition}

Note that the decay of $\theta$ in the regime $\alpha \in (0,1)$ is anomalous. Combining Theorem \ref{t:main} with Proposition \ref{p:per2} yields the first order scaling of the largest gap:

\begin{corollary}
\label{c:large}
Let $f$ be a $C^2$-smooth SGP with a spectral density. Then with $\zeta > 0$ as in Proposition \ref{p:per2}:
\begin{enumerate}
\item If ($\infty$-PD) holds, and moreover $\int K(x) \, dx \neq 0$, then as $R \to \infty$
\[ \frac{ \calL_R  }{\log R}  \to \frac{1}{\zeta} \qquad \text{in probability.} \]
\item If ($\alpha$-PD) holds for some $\alpha \neq 1$, $K > 0$, and $\sup_x |K'(x)| / K(x) < \infty$, then as $R \to \infty$:
\begin{enumerate}
\item If $\alpha > 1$,
\[ \frac{ \calL_R  }{\log R}  \to \frac{1}{\zeta} \qquad \text{in probability.} \]
\item If $\alpha \in (0,1)$,
\[ \frac{ \calL_R (\log \log R)^{1/\alpha}  }{(\log R)^{1/\alpha }} \to \Big( \frac{\alpha}{\zeta} \Big)^{1/\alpha}  \qquad \text{in probability.} \]
\end{enumerate}
\end{enumerate}
\end{corollary}

It is natural to expect that the decreasing function $|G'| \to 0$ defines a distribution that is in the domain of attraction of the Gumbel law in the sense of extreme value theory. In particular under ($\infty$-PD) or ($\alpha$-PD) with $\alpha >1$, one might expect that
\begin{equation}
\label{e:refas}
 \theta(r ) = \zeta r + o(1)   \, \qquad \text{as } r \to \infty.
\end{equation}
If this were true, we could deduce from Theorem \ref{t:main} that 
\begin{equation}
\label{e:ev}
  b_R (\calL_R -  a_R) \Longrightarrow \mathcal{X} \qquad \text{in law,} 
  \end{equation}
where $a_R = (1/\zeta) \log R$, $b_R = \zeta$, and $\mathcal{X}$ has a Gumbel distribution. Similarly, under ($\alpha$-PD) with $\alpha < 1$, one might expect that
\begin{equation}
\label{e:refas2}
 \theta(r ) = \zeta r^\alpha (\log r) + o(1)   \, \qquad \text{as } r \to \infty.
\end{equation}
If this were true, we could deduce from Theorem \ref{t:main} that \eqref{e:ev} holds with
\[ a_R =  \Big( \frac{\alpha}{\zeta} \Big)^{1/\alpha}   \frac{ (\log R)^{1/\alpha } } { (\log \log R)^{1/\alpha}  } \qquad \text{and} \qquad b_R =  \alpha^{1-1/\alpha}\zeta^{1/\alpha}(\log \log R)^{1/\alpha}  (\log R)^{1-1/\alpha} .   \]
However establishing \eqref{e:refas} or \eqref{e:refas2} seems very challenging, and to our knowledge even the analogous statements for $-\log G(r)$ are not known for any smooth SGP.

\subsection{Sketch of proof and outline of the paper}
Heuristically the Poisson approximation of the largest gaps follows from two basic ingredients: (i) the `splitting' of gap probabilities beyond mesoscopic scales, and (ii) the absence of clustering of large gaps.

\smallskip To be slightly more precise, let $L_R$ denote the order of the largest gap (e.g.\ $L_R= \frac 1\zeta\log R$ under ($\infty$-PD) as in Corollary \ref{c:large}), 
and let $s_R$ be some mesoscopic scale satisfying $L_R \ll s_R \ll R$; we will take $s_R = R^{\delta}$ for some carefully chosen $\delta > 0$.

\smallskip In Section \ref{s:split} we establish the `splitting' property: gaps of order $L_R$ at distance $\gg s_R$ appear approximately independently with multiplicative error $\ll 1$. Proving such a bound is relatively simple under ($\infty$-PD), but much harder under ($\alpha$-PD), especially if $\alpha$ is small. This error estimate
constitutes one of the main technical contributions of the paper.

\smallskip
In Section \ref{s:clus} we establish the absence of clustering: the probability of gaps in two intervals of order $L_R$ is at most $R^{-1-\eta}$ for some $\eta > 0$. We deduce this by combining the aforementioned results on the gap probability \cite{ffm25,ffm26} with some general comparison and decoupling techniques for Gaussian processes.

\smallskip In Section~\ref{s:pois}, we show how these ingredients combine to establish Theorem~\ref{t:main}. Our approach is based on the moment method in \cite{fw25}, but with several innovations:
\begin{itemize}
    \item Our strategy and techniques are very general, applying to a broad class of point processes (see Section~\ref{ss:gen}), whereas the proof in \cite{fw25} relies crucially on the special determinantal structure of eigenvalues. 
    \item In \cite{fw25} a Poisson approximation was only established for the rescaled largest gaps, whereas here we prove the joint convergence of the locations and size of the largest gaps, providing a more complete description of the limiting behaviour. 
\end{itemize}

\smallskip Some preliminary results are stated in Section \ref{s:prelim}, including general properties of SGPs, as well as the asymptotics of the scaling function in Proposition \ref{p:per2}.

\smallskip
Throughout the proof we make frequently use of the observation that $\mathcal{G}_I$ is the disjoint union of the \textit{persistence events} $\{f(x) > 0 : x \in I\}$ and $\{f(x) < 0 : x \in I\}$. These events have the advantage of being \textit{monotone} with respect to the process, which gives access to many tools specific to monotone events.


\subsection{Remarks on the decay condition}
We believe that the Poisson approximation is valid under much milder conditions than we establish, but not in full generality. In particular we believe that poly-logarithmic correlation decay does \textit{not} suffice:

\begin{conjecture}
\label{c}
There exists a $\gamma_0 > 0$ such that the Poisson approximation is valid for every smooth SGP satisfying $K(r) (\log r)^{\gamma'} \to 0$ as $r \to \infty$ for some $\gamma' > \gamma_0$, but not necessarily if $K(r) (\log r)^{\gamma'} \to \infty$ for some $\gamma' < \gamma_0$.
\end{conjecture}

Based on rough heuristics, we have reason to believe that the correct value of $\gamma_0$ is the golden ratio $(1+\sqrt{5})/2$. This would be rather surprising, especially when contrasted with the fact that the Poisson approximation is valid for \textit{high exceedences} of SGPs if $K(x) (\log x)^{\gamma} \to 0$ for $\gamma = 1$, but not necessarily otherwise \cite{llr83}, and is valid for \textit{smallest gaps} between SGP zeros with \textit{arbitrary} decay of correlations \cite{fgy24}.
 
 \smallskip
 Let us sketch these heuristics. As described in \cite{ffm26}, for a strongly correlated SGP with $K(x) \approx (\log x)^{-\gamma}$, a large gap event $\mathcal{G}_{[0,r]}$ is overwhelmingly due to the SGP experiencing `entropic repulsion' to a level $\ell_r  \gg 1$ which is the solution to the equation $\ell_r^2 / 2 = \log r $, i.e.\ $\ell_r = \sqrt{2 \log r}$, and the gap probability is then roughly $G(r) \approx e^{- (\ell_r^2/2) (\log r)^\gamma} = e^{- (\log r)^{\gamma + 1}}$.  If the Poisson approximation were valid, and supposing that $-G' \approx G$,  the relevant scaling function $\theta = -\log (-G')$ would therefore be $\theta(r) \approx (\log r)^{\gamma + 1}$. This would imply that the largest gap $\mathcal{L}_R$ in $[0,R]$ is approximately the solution to 
 \begin{equation}
     \label{e:heu1}
 (\log \mathcal{L}_R)^{\gamma+1} \approx \log R + \mathcal{X} 
 \end{equation}
 with $\mathcal{X}$ representing the (Gumbel) fluctuations.
 
 \smallskip
 Now assume that the SGP is not centred but instead has mean $\mu \approx 0$. Then the above heuristics must be modified: now $\ell_r$ is the solution to the equation $(\ell_r-\mu)^2 / 2 = \log r $ so that  $\ell_r^2 \approx 2 \log r + \mu \ell_r \approx 2 \log r + \mu \sqrt{2 \log r}$, and the gap probability is roughly $G(r) \approx e^{- (\ell_r^2/2) (\log r)^\gamma} \approx e^{- (\log r)^{\gamma + 1} - (\mu / \sqrt{2}) (\log r)^{\gamma + 1/2} }$. The largest gap  $\mathcal{L}_R$ would now be the solution to \begin{equation}
     \label{e:heu2}
  (\log \mathcal{L}_R)^{\gamma + 1} - (\mu / \sqrt{2}) (\log \mathcal{L}_R)^{\gamma + 1/2}  \approx \log R + \mathcal{X}. 
   \end{equation}
   If $\mu$ was in fact a \textit{random} constant, comparing \eqref{e:heu1} and \eqref{e:heu2} suggests that the fluctuations in $\mu$ would influence those of $\mathcal{L}_R$ if and only if
  \begin{equation}
      \label{e:heu3}
    1 \lesssim \mu  (\log \mathcal{L}_R)^{\gamma+1/2}  \approx \mu (\log R)^{(\gamma+1/2)/(\gamma+1)}  .
    \end{equation}

 \smallskip
 Finally observe that a SGP with $K(x) \approx (\log x)^{-\gamma}$ has `global fluctuations' on $[0,R]$ of order $(\log R)^{-\gamma/2}$, whose effect can be approximated by a random level shift of this order. In view of \eqref{e:heu3}, this suggests that the correlations interfere with the Poisson approximation if and only if 
  \[1 \lesssim (\log R)^{-\gamma/2} (\log R)^{(\gamma+1/2)/(\gamma+1)}  . \]
  Solving $-\gamma/2 + (\gamma+1/2)/(\gamma+1) = 0$, we find that the critical value of $\gamma$ is the golden ratio.
  
\subsection{Acknowledgments}
S.M.\ is supported by the Australian Research Council Future Fellowship FT240100396. The authors thank Michael McAuley for helpful discussions.

\medskip
\section{Preliminaries}
\label{s:prelim}

In this section we establish some preliminary properties of Gaussian processes, the gap probability $G(r) = \P[\mathcal{G}_{[0,r]}]$, and the probability of persistence events.

\subsection{Basic properties of Gaussian processes}

We collect some properties of smooth SGPs. The results are standard, but we give details of the proof for completeness.  

\begin{lemma}[Non-degeneracy]
\label{l:nondegen}
Let $f$ be a $C^1$-smooth SGP with a spectral density. Then for every finite collection of disjoint points $(x_i)_{i \le k}$, the Gaussian vector
\[ \big( f(x_1), \ldots , f(x_k), f'(x_1), \ldots,  f'(x_k) \big) \]
is non-degenerate.
\end{lemma}
\begin{proof}
The vector is degenerate if and only if there exist $\lambda = (\lambda_i) \in \R^{2k} \setminus \{0\}$ such that $\sum_{j \le k} \lambda_j f(x_j) + \sum_{j \le k} \lambda_{j+k} f'(x_{j})$ has zero variance; in terms of the spectral measure $\mu$ this condition is
\[  \int  \Big| \sum_{j \le k} \lambda_j e^{  i x_j s } +  \sum_{j \le k} \lambda_{j+k}  (i  s)   e^{i x_{j} s } \Big|^2  \, d \mu(s)  = 0 . \]
Suppose such a $\lambda$ exists. Since $t(s) := \sum_{j \le k} \lambda_j e^{i x_j s } +  \sum_{ j \le k} \lambda_{j+k}  i  s   e^{i x_{j} s }$ has a countable set of zeros, $\mu$ must be supported on a countable set, which is inconsistent with having a density.
\end{proof}

\begin{lemma}[Decay]
\label{l:m}
 Let $f$ be a $C^1$-smooth SGP with a spectral density. Then  as $x \to \infty$
 \[ \max\{|K(x)|,|K'(x)|,|K''(x)| \} \to 0.\]
 \end{lemma}
\begin{proof}
Let $\rho$ denote the spectral density. Since $f$ is $C^1$-smooth, $K$ is in $C^2$, and so each of the densities $\rho dx$, $|x| \rho dx$, and $x^2 \rho dx$ are in $L^1$. The conclusion follows from the Riemann-Lebesgue lemma.
\end{proof}

\begin{lemma}[Tail bound]
\label{l:derbound}
Let $f$ be a $C^1$-smooth SGP. Then there exists $c_1,c_2 > 0$ such that if $r \ge 2$ and $s>c_1\sqrt{\log r}$, 
\[ \P \Big[ \sup_{x \in [0,r]}  |f'(x)| > s \Big] \le e^{-c_2 s^2 } .\]
\end{lemma}
\begin{proof}
This is a simple consequence of the union bound and the Borell-TIS inequality applied to the process $(f'(x))_{x \in [0,1]}$.
\end{proof}

\subsection{Persistence events}

For an interval $I$ and $\sigma \in \{-,+\}$ define the \textit{persistence events}
\[ \mathcal{P}^\sigma_I = \{ \sigma f(x) > 0: x \in I \} =  \begin{cases} 
 \{ f (x) > 0 : x \in I \} & \text{if } \sigma = + \\ \{ f(x) < 0 : x \in I \} & \text{if } \sigma = - \end{cases} ,\]
and observe that $\mathcal{G}_I$ is the disjoint union of $\mathcal{P}_I^+$ and $\mathcal{P}_I^-$. 

\smallskip
There has been much recent work devoted to the study of the persistence probability $$P(r): = \P[\mathcal{P}_{[0,r]}^+]=G(r)/2$$ and the analogous quantities for non-centred processes. The following result summarises findings in \cite{ffm25, ffm26}:
\begin{proposition}[Persistence log-asymptotics {\cite{ffm25, ffm26}}]
\label{p:per}
Let $f$ be a $C^2$-smooth SGP with a spectral density.
\begin{enumerate}
\item If either ($\infty$-PD) is satisfied or ($\alpha$-PD) is satisfied with $\alpha > 1$, and $\int K(x) \,dx \neq 0$, then there exists $\zeta > 0$ such that, for any sequence $\eps_r$ satisfying $\eps_r \to 0$ as $r \to \infty$,
\[ -\log \P \big[ f(x) > -\eps_r : x \in [0,r] \big]   \sim \zeta r  \ , \quad \text{as } r \to \infty . \]
\item If ($\alpha$-PD) is satisfied with $\alpha \in (0,1)$, then there exists $\zeta > 0$ such that, for any $\eps \in \R$,
\[-\log \P \big[ f(x) > -\eps : x \in [0,r] \big]  \sim \zeta r^\alpha (\log r)  \ , \quad \text{as } r \to \infty . \]
 In this case $\zeta$ has the explicit form
 \[ \zeta = \frac{K(0) \sqrt{\pi} (1-\alpha)  }{\Gamma((1-\alpha)/2) \Gamma(1+\alpha/2) } =  \frac{2 K(0) (1-\alpha) \sin(\alpha \pi /2) }{ \sqrt{\pi} \alpha} . \]
\end{enumerate}
\end{proposition}
\begin{proof}
(1). \cite[Theorem 1]{ffm25} states that, under more general conditions than we assume, for every $\eps \in \R$ there is a $\zeta_\eps > 0$ such that
\[ -\log \P \big[ f(x) > -\eps : x \in [0,r] \big]   \sim \zeta_\eps r  \ , \quad \text{as } r \to \infty . \]
\cite[Theorem 4]{ffm25} states that, again under more general conditions than we assume, the map $\eps \mapsto \zeta_\eps$ is continuous. Taking $\zeta := \zeta_0$, our conclusion follows by combining these statements.

(2). \cite[Theorem 3]{ffm26} states this result for $\eps = 0$, and see the comments in \cite[Section 1.6]{ffm26} for the extension to arbitrary $\eps \in \R$ (in fact any  $\eps_r$ satisfying $\eps_r / \sqrt{\log r} \to 0$).
\end{proof}

We state a well-known comparison result for persistence events:

\begin{proposition}[Slepian's lemma]
Let $(X_i)_{i\le n}$ and $(Y_i)_{i \le n}$ be centred Gaussian vectors satisfying, for all $1 \le i,j \le n$,
\[ \E[X_i^2] = \E[Y_i^2] \qquad \text{and} \qquad  \E[X_i X_j] \le \E[Y_i Y_j] . \]
Then for every $\ell \in \R$
\[\P \big[ X_i < \ell \ \forall i\big] \le \P \big[ Y_i < \ell  \ \forall i\big] .\]
\end{proposition}

By approximation the conclusion of Slepian's lemma extends in the natural way to continuous Gaussian processes indexed by (a finite union) of compact intervals.

\smallskip
We shall also make use of the following `sprinkled decoupling' inequality for persistence events, recently proven in \cite{mui23}:

\begin{proposition}[{\cite[Theorem 1.10]{mui23}}]
\label{p:sdi}
\label{t:sdi2}
Let $(X_i)_{i\le n}$ be a centred Gaussian vector, and abbreviate  $ \|\sigma^2\|_\infty := \max_i \textrm{Var}[X_i]$. Then for all $I_1,I_2 \subseteq \{1,\ldots,n\}$, $\ell \in \R$, and $\eps > 0$,
\begin{align*}
 &  \P[X_i > \ell \ \forall i \in I_1  \cap I_2 ]  \\
& \qquad \le  \P[X_i  > \ell \ \forall i \in I_1]\P[X_i > \ell -  \eps  \ \forall i \in I_2] +   \exp \Big(- \frac{ \eps^2 }{8 \|\sigma^2\|_\infty  \rho^2(I_1,I_2) } \Big) 
\end{align*}
 where 
 \begin{equation}
\label{e:mcc}
 \rho(I_1,I_2)  = \sup_{ \alpha \in \R^{|I_1|} , \beta \in \R^{|I_2|} }   \frac{ | \textrm{Cov}[  \langle \alpha , X_{I_1} \rangle , \langle \beta , X_{I_2} \rangle ] | }{ \sqrt{ \textrm{Var}[\langle \alpha , X_{I_1}  \rangle ] \textrm{Var}[ \langle \beta , X_{I_2}  \rangle] } } 
 \end{equation}
 is the \textit{maximum correlation coefficient}.
\end{proposition}

\subsection{The gap probability and its derivatives}
 
We derive basic properties of $G$ and give Kac-Rice-type formulae for $G$ and its derivatives:

\begin{lemma}
\label{l:lambda}
Let $f$ be a $C^2$-smooth SGP with a spectral density. Then $G$ is convex and twice-differentiable. Assuming that $f$ has unit variance for simplicity, the derivatives can be expressed as
\begin{equation}
\label{e:gprime}
- G'(r) =  \E \big[  \big| \{  z \in [0,1] \cap \mathcal{Z} : f(z + \cdot) \in \mathcal{G}_{(0,r)}  \} \big| \big]  = (2\pi)^{-1/2} \E \big[ | f'(0)|  \id_{\mathcal{G}_{(0,r)} }   \big| f(0) = 0 \big] , 
\end{equation}
and
\begin{equation}
\label{e:gprimeprime}
G''(r) = (2\pi)^{-1}  ((1 - K(r)^2)^{-1/2}  \E \big[ | f'(0) f'(r) | \id_{\mathcal{G}_{(0,r)} }   \big| f(0) = f(r) = 0 \big]  .
\end{equation}
\end{lemma}

\begin{proof}
In the proof we will repeatedly use the non-degeneracy provided by Lemma \ref{l:nondegen} without explicit reference. For an interval $I \subset \R$ we abbreviate $N_I = |\{ z \in I  \cap \mathcal{Z} \} |$. For a non-degenerate Gaussian vector $X$, let $\varphi_{X}(0)$ denote its density at zero. We first claim that
\begin{equation}
    \label{e:ll1}
 -G'(r) = \lim_{\delta \to 0}  \delta^{-1} \E \big[ N_{[-\delta,0]} \id_{\mathcal{G}_{(0,r)} }  \big]  
 \end{equation}
and
\begin{equation}
    \label{e:ll2}
G''(r)= \lim_{\delta \to 0}  \delta^{-2} \E \big[ N_{[-\delta,0]} N_{[r,r+\delta]} \id_{\mathcal{G}_{(0,r)} }  \big] .   
\end{equation}
Assuming this, let us complete the proof. By the Kac-Rice formula \cite[Theorem 11.2.1]{at07}, for every $\delta > 0$ we have
\begin{align}
\label{e:ll3}
 \E \big[  N_{[-\delta,0]} \id_{\mathcal{G}_{(0,r)} }  \big]  = \int_{-\delta}^{0}  \E \big[ | f'(x)| \id_{\mathcal{G}_{(0,r)} }   \big| f(x) = 0 \big] \varphi_{f(x)}(0) \, dx  .
  \end{align}
(Strictly speaking, one should apply \cite[Theorem 11.2.1]{at07} after replacing the event $\mathcal{G}_{(0,r)}$ with the finite-dimensional event  $\mathcal{P}^k_{(0,r)} = \{f(x_i) > 0 : x_i \in (0,r) \cap (\mathbb{Z}/k) \}$, and conclude by dominated convergence.) Similarly the Kac-Rice formula also gives
\begin{align}
\label{e:ll4}
 & \E \big[  N_{[-\delta,0]} N_{[r,r+\delta]} \id_{ \mathcal{G}_{(0,r)} }  \big] \\
\nonumber & \qquad \qquad = \int_{-\delta}^0 \int_{r}^{r+\delta}  \E \big[ | f'(x)| |f'(y) | \id_{\mathcal{G}_{(0,r)} }   \big| f(x) = f(y) = 0 \big] \varphi_{f(x),f(y)}(0) \, dy dx .
  \end{align}
Since the integrands in \eqref{e:ll3} and \eqref{e:ll4} are continuous functions of $x < 0$ and $y > r$, taking the $\delta \to 0$ limit we obtain
\[ - G'(r) = (2\pi)^{-1/2} \E \big[ | f'(0)|  \id_{\mathcal{G}_{(0,r)} }   \big| f(0) = 0 \big] \]
and
\[ G''(r) = (2\pi)^{-1}  ((1 - K(r)^2)^{-1/2}  \E \big[ | f'(0) f'(r) | \id_{\mathcal{G}_{(0,r)} }   \big| f(0) = f(r) = 0 \big] .\]
Finally, the Kac-Rice formula also gives
\[ \E \big[  \big| \{  z \in [0,1] \cap \mathcal{Z} : f(z + \cdot) \in \mathcal{G}_{(0,r)}  \} \big| \big]  = (2\pi)^{-1/2} \E \big[ | f'(0)|  \id_{\mathcal{G}_{(0,r)} }   \big| f(0) = 0 \big] \]
which completes the proof. 

It remains to establish \eqref{e:ll1} and \eqref{e:ll2}. First observe that
\[ -G'(r) =  \lim_{\delta \to 0} \delta^{-1} \P[ \mathcal{G}_{[0,r+\delta]} \setminus  \mathcal{G}_{[0,r]} ] =  \lim_{\delta \to 0} \delta^{-1} \P \big[ N_{[-\delta,0]} \ge 1 , \mathcal{G}_{(0,r)}   \big]    \]
where we used stationarity, and the fact that the events $\mathcal{G}_{(0,r)} $ and $\mathcal{G}_{[0,r]}$ coincide up to a null set. Similarly
\[ G''(r) =\lim_{\delta \to 0} \delta^{-2} \P \big[ N_{[-\delta,0]} \ge 1 , N_{[r,r+\delta]} \ge 1, \mathcal{G}_{(0,r)}   \big]    . \]
Using the fact that
\[  \big| \P[N \ge 1, M \ge 1, E] - \E[N M \id_E ] \big| \le \E[N(N-1)M] + \E[N M(M-1)]  \]
for arbitrary random variables $N,M \in \mathbb{N}$ and event $E$, it remains to show that 
\[ \E \big[ N_{[-\delta,0]} (N_{[-\delta,0]}-1) \big] = O( \delta^2) \qquad \text{and} \qquad \E \big[ N_{[-\delta,0]} (N_{[-\delta,0]}-1) N_{[r,r+\delta]} \big]  = O(\delta^3)  . \] 
By the Kac-Rice formula, these follows from the fact that (see, e.g., \cite[Theorem 1.13]{al21}) the two-point density of zeros
\[ \rho_2(x,y) =  \E \big[ | f'(x) f'(y) | \big| f(x) = f(y) = 0 \big] \varphi_{f(x),f(y)}(0)  \]
is bounded over $x \neq y$, while the three-point density 
\[ \rho_3(x,y,z) = \E \big[ | f'(x) f'(y) f'(z) | \big| f(x) = f(y) = f(z) = 0 \big] \varphi_{f(x),f(y),f(z)}(0)  \]
is bounded on $(x,y,z) \in ([-\delta,0]^2 \times [r,r+\delta]
)\setminus \{x=y\}$. (More precisely, \cite[Theorem 1.13]{al21} states that $\rho_3$ is bounded outside $\{x=y\} \cup \{x=z\} \cup \{y=x\}$ under the assumption that $f$ is $C^3$-smooth; in our case it is sufficient that $f$ is $C^2$-smooth since we assume $z$ is bounded away from $(x,y)$.)
\end{proof}

Combining the above formulae with Proposition \ref{p:per}, we derive asymptotic bounds on the derivatives of $G$:

\begin{proposition}
\label{p:gpp}
Under the conditions of Proposition \ref{p:per2} (either item), as $r \to \infty$
\[ \theta(r) := - \log (-G'(r))  \sim   - \log G(r)  \qquad \text{and} \qquad   \log G''(r)  \le    \log G(r)(1 + o(1)) . \]
\end{proposition}

\begin{proof}
Let us begin by establishing the upper bound on $\theta(r)$. By the convexity of $G$, 
\[  G(2r) \ge G(r) + r G'(r) \implies |G'(r)| \le \frac{G(r) - G(2r) }{r} .\]
Suppose the conditions in the first item of Proposition \ref{p:per} hold so that $G(r) = e^{-\zeta r + o(r)}$.  Then as $r \to \infty$
\[ |G'(r)| \le \frac{G(r) - G(2r) }{r} = \frac{e^{-\zeta r + o(r)} - e^{ - 2 \zeta r + o(r)}}{r} = e^{-\zeta r + o(r) } \]
which implies that 
\[ \theta(r) = - \log |G'(r)|  \le \zeta r + o(r) \sim - \log G(r) .\] 
Suppose instead that the conditions in the second item of Proposition \ref{p:per} hold so that $G(r) = e^{-\zeta r^\alpha \log r + o(r^\alpha \log r)}$. Then the same argument gives that
\[ \theta(r) = - \log |G'(r)|  \le \zeta r^\alpha (\log r) + o(r^\alpha (\log r) ) \sim - \log G(r) \]
also in that case. 



We turn to lower bounds on $\theta(r)$. Assume $K(0) =1$ for simplicity. Fix $\delta > 0$. Applying H\"{o}lder's inequality to \eqref{e:gprime}, and since  $f(0)$ and $f'(0)$ are independent,
\begin{align}
\nonumber -G'(r) & \le   (2\pi)^{-1/2}  \E \big[ |f'(0)|^{(1+\delta)/\delta}  \big]^{\delta/(1+\delta)} \P[ \mathcal{G}_{(0,r)}  | f(0) = 0 ]^{1/(1+\delta)}  \\
\label{e:lam} & = c_\delta \P[  \mathcal{G}_{(0,r)}  | f(0) = 0 ]^{1/(1+\delta)} .
\end{align}
It remains to bound $\P[  \mathcal{G}_{(0,r)} | f(0) = 0 ]$, which by symmetry is $2 \P[ f(x)  > 0  : x \in [0,r] | f(0) = 0]$. Fix $\eps > 0$, and for $s\geq 0$ define $\Psi(s) :=  \P[ f(x) > - \eps : x \in [0,r]   | f(0) =s]$. By Gaussian regression and the fact that $|K(x)|\leq K(0)=1$, 
\[ \E[f(x) | f(0) = s ] =  s K(x)  \ge - s K(0) = -s, \]
and so by monotonicity, if $s \in [0, \eps]$, 
\begin{align*}
\Psi(s) & \ge \P \big[ f(x)  > - \eps - s K(x) : x \in [0,r]  \big| f(0) = 0 \big] \\
 & \ge \P \big[ f(x)  > - \eps +s : x \in [0,r]  \big| f(0) = 0 \big] \\ 
 & \ge \P \big[   \mathcal{G}_{(0,r)}  \big| f(0) = 0 \big] .  
 \end{align*}
 On the other hand,
\begin{align*}
  \min_{s \in [0, \eps ]} \Psi(s) & \le \P \big[ f(x) > - \eps : x \in [0,r]   \big| f(0) \in [0,\eps] \big]   \\
  & \le  \frac{ \P[ f(x) > - \eps : x \in [0,r]  ] }{  \P[ f(0) \in [0, \eps ] ] } = c_\eps \P[ f(x) > - \eps : x \in [0,r]  ]   .
  \end{align*}
Combining with \eqref{e:lam} we have
\[-G'(r) \le c_{\delta,\eps} \big( \P \big[ f(x) > - \eps : x \in [0,r] \big]    \big)^{1/(1+\delta)}  .\]

Assume now the conditions in the first item of Proposition \ref{p:per}. Then applying the first item of Proposition \ref{p:per}, for every $\eps' > 0$ there exists a $\eps > 0$ sufficiently small such that
\[ \P[ f(x) > - \eps : x \in [0,r]  ] \le e^{-(\zeta-\eps')r } \]
for sufficiently large $r > 0$. We conclude that, for every $\delta > 0$ and $\eps' > 0$
\[-G'(r) \le  c_{\delta,\eps'}  e^{- ( (\zeta-\eps') / (1+\delta) ) r } \]
for sufficiently large $r > 0$, which proves the result by taking $\delta,\eps' \to 0$. The proof under the conditions in the second item of Proposition \ref{p:per} is identical, except that we apply the second item in Proposition \ref{p:per}. 

We turn to the bound on $G''(r)$. As in the previous proof it suffices to show that, for every $\delta,\eps > 0$,
\begin{equation}
\label{e:lam3}
G''(r) \le c_{\delta,\eps} \big( \P \big[ f(x) > - \eps : x \in [0,r]  \big]    \big)^{1/(1+\delta)}  . 
\end{equation}
Assume $K(0) =1$ for simplicity.  Fix $\delta > 0$. Applying H\"{o}lder's inequality to \eqref{e:gprimeprime},
\[ G''(r) \le   (2\pi)(1 - K(r)^2)^{-1/2}  \E \big[ |f'(0) f'(r) | ^{(1+\delta)/\delta} | f(0) = f(r) = 0  \big]^{\delta/(1+\delta)} \] \[\times \P[ \mathcal{G}_{(0,r)}  | f(0) = f(r) = 0 ]^{1/(1+\delta)}   .\]
By Gaussian regression, the conditional variance satisfies 
\[ \mbox{Var}[f'(0) |  f(0) = f(r) = 0 ] = -K''(0)-\frac{[K'(r)]^2}{1-K(r)^2} \le  -K''(0) . \] 
with the same estimate for 
$\mbox{Var}[ f'(r) | f(0) = f(r) = 0 ] $ by symmetry. By the Cauchy-Schwarz inequality, this implies the following uniform upper bound 
\[ \E \big[ |f'(0) f'(r) | ^{(1+\delta)/\delta}|f(0)=f(r)=0  \big] \le  c_\delta  . \] 
Since also $K \to 0$ we have, for $r$ sufficiently large,
\begin{equation}
\label{e:lam2}
G''(r) \le  c''_\delta  \P[ \mathcal{G}_{(0,r)}  | f(0) = f(r) = 0 ]^{1/(1+\delta)}  . 
\end{equation}
It remains to bound $\P[  \mathcal{G}_{(0,r)} | f(0) = f(r) = 0 ] = 2 \P[ f(x)  > 0  : x \in [0,r] | f(0) = f(r) = 0]$. Fix $\eps > 0$, and for $s,t \in \R$ define $\Psi(s,t) :=  \P[ f(x) > - \eps : x \in [0,r]   | f(0)  = s, f(r) = t]$. By Gaussian regression, 
\begin{align*}
 p(x;s,t) & := \E[ f(x) | f(0) =s, f(r) = t ] \\
 & = \frac{1}{1-K(r)^2} \Big( s ( K(x) - K(x-r)K(r) ) + t( K(x-r) - K(x) K(r) ) \Big)  \\
 & \ge  \frac{-2 (|s| + |t|)}{1 - K(r)^2}
 \end{align*}
and so by monotonicity, if $r$ is sufficiently large so that $K(r)^2 \le 1/2$ and $s,t \in [0, \eps/8]$, 
\begin{align*}
 \Psi(s, t) & \ge \P \big[ f(x)  > - \eps - p(x;s,t)  : x \in [0,r]  \big| f(0) =f(r)= 0 \big] \\
 & \ge \P \big[ \mathcal{G}_{(0,r)} \big| f(0) = f(r) = 0 \big] .
 \end{align*}
 On the other hand,
\[ \min_{s ,r \in [0, \eps/8]} \Psi(s,t) \le \frac{ \P[ f(x) > - \eps : x \in [0,r]  ] }{  \P[ f(0),f(r) \in [0, \eps/8] ] } \le c_\eps \P[ f(x) > - \eps : x \in [0,r]  ]   \]
where the last inequality used that $K \to 0$. This yields \eqref{e:lam3}, and thus we complete the proof of the proposition. 
 \end{proof}

\medskip
\section{Splitting of gap probabilities}
\label{s:split}

 In this section we show that gap events on well-separated intervals are approximately independent. We begin by considering collections of intervals which are pairwise well-separated. In that case we provide estimates with two-sided multiplicative error. In the following subsection we generalise the set-up to allow for `clustering' of intervals, but only prove weaker decoupling bounds.

\subsection{Splitting with multiplicative error}
Recall that $\mathcal{G}_I$ denotes the event that $I \cap \mathcal{Z}$ is empty. Let $\mathcal{Q}^k_{r,s}$ denote the set of collections of $k$ disjoint compact intervals $(I_i)_{i \le k}$ such that each of their lengths is at most $r$ and they are pairwise separated by distance at least $s$. For $k \ge 2$ define the splitting coefficient
\begin{equation}
\label{e:rho}
 \rho^k_{r,s} = \inf \Big\{ \eps \in [0,1] :   \frac{  \P[ \cap_i \calG_{I_i} ] }{ \prod_i \P[  \calG_{I_i} ] }  \in [ 1-\eps, 1/(1-\eps) ] \quad \text{for all }   \ (I_i)_{i \le k} \in \mathcal{Q}^k_{r,s}  \Big\}  . 
 \end{equation}
We establish conditions which ensure that $\rho^k_{r,s}$ is small if $s \gg r \gg 1$. We state two versions of this result. The first gives weaker bounds under general conditions, and will be sufficient for our purposes under the decay assumption ($\infty$-PD). The second requires more conditions but yields a bound which is strong enough to cover the case ($\alpha$-PD).

\smallskip

We begin by stating the weak version of the bound. Recall that $G(r) = \P[\mathcal{G}_{[0,r]}]$, and define $\bar{K}(s):=\max_{r\geq s}|K(r)|$.

\begin{proposition}[Splitting, weak version]
\label{p:wsplit}
Let $f$ be a $C^2$-smooth centred SGP with a spectral density. Then for every $k \ge 2$ there exists a $c > 0$ such that for every $r,s \ge 1$
\[ \rho^k_{r,s} \le c  r^2  \bar{K}(s)  / G(r)^k  .\]
\end{proposition}

The stronger version of the bound requires extra conditions. These are only used in one place in the proof (see Claim \ref{c:h} below), and we expect they could be weakened. 

\begin{conditions}
\label{c:ss}
Suppose that $K(x) > 0$, $\sup_x |K'(x)| / K(x) < \infty$, and for every $c_1 > 0$ there exists $c_2 > 0$ such that, for every $r \ge 1$
\begin{equation}
\label{e:decay}
\min_{|x| \le r} K(x) \ge c_1 \bar{K}(c_2 r) .
\end{equation}
\end{conditions}

\noindent We note that \eqref{e:decay} is satisfied if $K(x) > 0$ and $K(x) \sim |x|^{-\alpha}$, $\alpha > 0$.

\begin{proposition}[Splitting, strong version]
\label{p:ssplit}
Suppose $f$ is as in Proposition \ref{p:wsplit}, and moreover Conditions \ref{c:ss} hold. Then for every $k \ge 2$ there exists a $c > 0$ such that for every $r \ge c$ and $s \ge c r$
\begin{equation}
\label{e:ss}
\rho^k_{r,s} \le   c  r^2  \bar{K}(s)   (\log ( 1/G(r)) )^3  .
\end{equation}
\end{proposition}
\begin{remark}
This improves on the weak version by replacing $G(r)^{-k}$ with $ (\log ( 1/G(r)) )^3 $. The exponent $3$ arises in the proof as $3 = 2d+1$ where $d=1$ is the dimension of the domain $\R$ of the process.
\end{remark}

We will apply these bounds to a finite number of gap events for which $r =  L_R$ is the order of the largest gap, $G(R)  \approx 1/R$, and $s \approx R^\delta$ is a mesoscopic scale. Hence for the first version to be effective we need $ L_R^2 \bar{K}(R^\delta) R^k \to 0$ for every $k$, which requires ($\infty$-PD), whereas the second version needs only that $ L_R^2 \bar{K}(R^\delta) (\log R)^3 \to 0$ which is satisfied under ($\alpha$-PD) for any $\alpha > 0$.

\smallskip
 In the regime of logarithmical decay $K(x) \sim (\log x)^{-\gamma}$, $\gamma > 0$, the second version is still not strong enough (even if we could remove the logarithmic factor), since in that regime the largest gaps are of order $L_R \approx e^{ (\log R)^{1/(\gamma + 1)} }$ so that $ L_R^2 K(R^\delta) \to \infty$. 

\medskip
To prove Propositions \ref{p:wsplit} and \ref{p:ssplit} it is sufficient to bound the analogous splitting coefficient for persistence events $\calP_I^{\pm}$. Define
 \[  \tilde{\rho}^k_{r,s} = \inf \Big\{ \eps \in [0,1] :   \frac{  \P[ \cap_i \calP^{\sigma_i}_{I_i} ] }{  \prod_i \P[  \calP^{\sigma_i}_{I_i} ]    }  \in [ 1-\eps,  1/(1-\eps)]  \quad \text{for all }  m \le k , \   (I_i) \in \mathcal{Q}^m_{r,s}, \ (\sigma_i) \in \{-,+\}  \Big\}  . \]
Then since $\calG_I$ is a disjoint union of  $\calP_I^+$ and $\calP_I^-$, it is easy to see that 
\begin{equation}
\label{e:persplit}
 \rho^k_{r,s} \le  \tilde{\rho}^k_{r,s} .
\end{equation}
We control $\tilde{\rho}^k_{r,s}$ by exploiting an interpolation formula for persistence events, described in the next section.

\subsection{Interpolation}
The interpolation formula is valid under much weaker conditions, so throughout this subsection we suppose that $f$ is a $C^2$-smooth centred Gaussian process such that the vector $(f(x),f(y),f'(x),f'(y))$ is non-degenerate for every $x \neq y$. 

\smallskip
To state the formula we introduce some definitions.  A \textit{clustering (of intervals)} is a finite collection of clusters $(\mathcal{C}_\ell)_{1 \le \ell \le m}$, each  containing a finite set of bounded intervals $(I_i)$, such that closures of $I_i \in \cup_{\ell} \mathcal{C}_\ell$ are mutually disjoint. A clustering $(\mathcal{C}_\ell)_{\ell \ge 1}$ is \textit{simple} if each cluster $\mathcal{C}_\ell$ contains a single interval. Propositions \ref{p:wsplit} and \ref{p:ssplit} concern simple clusterings; we introduce the general set-up with a view to applications in the final subsection.

\smallskip
 Let $(\mathcal{C}_\ell)_{\ell \ge 1}$ be a clustering of bounded intervals $(I_i)_{i \ge 1}$. For $t \in [0,1]$, define the process $f_t(z)$ on the collection of clusters  as follows: for $z$ in the cluster $\mathcal C_\ell$, i.e., $z \in \cup_{I_i \in \mathcal{C}_\ell} I_i$ then
\[  f_t(z) = \sqrt{t} f(z) + \sqrt{1-t} \tilde{f}_\ell(z) \]
where, for each $\ell \ge 1$, $\tilde{f}_\ell$ is an independent copy of the restriction $f|_{\cup_{I_i \in \mathcal{C}_\ell} I_i}$. A direct computation shows that
\begin{equation}\label{covwitht}
\mathbb{E}[f_t(x)f_t(y)] =
\begin{cases}
K(x,y), & x,y \mbox{ are in the same cluster},\\
t\,K(x,y), &x,y \mbox{ are in two different clusters.}
\end{cases}
\end{equation}

 Associate to every $I_i \in \mathcal{C}_\ell$ a sign $\sigma_i \in \{-,+\}$. Let $A_t$ denote the event $\{f_t  \in \cap_i  \calP^{\sigma_i}_{I_i}\}$, and let $A_t^{xy}$ be the modification of $A_t$ for which, at points $x \in I_i$ and $y \in I_j$, we relax the condition $\{ \sigma_i f(x) > 0 , \sigma_j f(y) > 0\}$ to its non-strict equivalent $\{ \sigma_i f(x) \ge 0 , \sigma_j f(y) \ge 0\}$.

\smallskip
Let $\mu_{I_i}$ be the measure defined as the sum of the Lebesgue measure on $\textrm{int}(I_i)$ and a Delta mass on each point of $\partial I_i$.

\begin{proposition}[Interpolation]
\label{p:inter}
The map $t \mapsto \P[A_t]$ is absolutely continuous on $[0,1]$ with a.e.\ derivative
 \begin{align}
 \label{e:inter}
   \frac{d}{dt} \P[A_t] =  \sum_{\ell < \ell'} \sum_{I_i \in \mathcal{C}_\ell, I_j \in \mathcal{C}_{\ell'}}  \sigma_i \sigma_j  \int \int K(x,y)  p_t(x,y) \, \mu_{I_i}(dx) \mu_{I_j}(dy)   
\end{align}
where, if $x \in \textrm{int}(I_i)$ and $y \in \textrm{int}(I_j)$,
\begin{equation}
\label{e:p}
 p_t(x,y) :=  \E \big[   |f_t''(x)| |f_t''(y)|   \id_{A_t^{xy}} \big| f_t(x) = f_t(y) = f_t'(x) = f_t'(y) = 0 \big] \tilde{\varphi}_t(x,y) ,
 \end{equation}
 and $\tilde{\varphi}_t(x,y)$ denotes the density at zero of the Gaussian vector $(f_t(x), f_t(y), f_t'(x), f_t'(y))$, whereas if $x \in \partial I_i$ (resp.\ $y \in \partial I_j$) the terms $f''_t(x)$ and $f'_t(x)$ (resp.\ $f''_t(y)$ and $f'_t(y)$) are removed from the integrand, conditioning, and Gaussian density in \eqref{e:p}. 

In particular, \eqref{e:inter} implies that
\begin{align}
    \label{e:intercon}
& \P \big[ \cap_i \calP^{\sigma_i}_{I_i} \big] -  \prod_{\ell \ge 1} \P \big[ \cap_{I_i \in \mathcal{C}_\ell} \calP_{I_i}^{\sigma_i}  \big]  \\
& \nonumber \qquad \qquad =  \int_0^1  \sum_{\ell < \ell'} \sum_{I_i \in \mathcal{C}_\ell, I_j \in \mathcal{C}_{\ell'}} \sigma_i \sigma_j  \int \int K(x,y)  p_t(x,y) \, \mu_{I_i}(dx) \mu_{I_j}(dy)  \, dt      .
\end{align} 
\end{proposition}

\begin{remarks}
In the case of a simple clustering, the conclusion \eqref{e:intercon} reads more simply
\[  \P \big[ \cap_i \calP^{\sigma_i}_{I_i} \big] -  \prod_{i \ge 1} \P \big[ \calP_{I_i}^{\sigma_i}  \big]   =  \int_0^1  \sum_{i <  j}  \sigma_i \sigma_j  \int \int K(x,y)  p_t(x,y) \, \mu_{I_i}(dx) \mu_{I_j}(dy)  \, dt      .\]
\end{remarks}
\begin{proof}
In the case of two clusters ($m=2$), this is an instance of a general covariance formula valid for all `topological events' of smooth Gaussian fields on general (stratified) manifolds \cite{bmr20}. We describe how to extend the proof to $m \ge 2$; in fact only the very first step in the proof needs to be modified.

The proof in \cite{bmr20} relies on the following classical finite-dimensional version of the interpolation formula (\cite[Lemma 2.22]{bmr20}, and see also \cite[Theorem 1.4]{piterbarg}). Let $(X^1_t, X^2_t)$ be a Gaussian vector in $\R^{2n }$ with covariance matrix 
\begin{equation}
    \label{e:cov}
 \textrm{Cov}(X^i_t,X_t^j) = \begin{cases} I_n&  \text{if } i = j, \\ t I_n & \text{if } i \neq j , \end{cases} 
 \end{equation}
where $I_n$ is the identity matrix. Let $A^i$, $i=1,2$, be domains in $\R^n$ whose boundaries are piecewise smooth, and which have surface areas, inside the ball of radius $R$, that grow
at most polynomially in $R$. Then $t \mapsto \mathbb{P}[X^1_t \in A^1, X^2_t \in A^2]$ is absolutely continuous on $[0,1]$ with a.e.\ derivative
\begin{equation}
    \label{e:pit}
 \frac{d}{dt}\mathbb{P}\left[X^1_t \in A^1, X^2_t \in A^2 \right]=\int_{\partial A^1 \times \partial A^2}\langle\nu_{A^1}(x_1),\nu_{A^2}(x_2)\rangle \gamma_t(x_1,x_2) \, d x_1 dx_2, 
 \end{equation}
 where $\nu_{A^i}$ is the outward unit normal vector on the boundary of $A^i$, $\int dx_i$ denotes integration with respect to the surface area measure on $\partial A^i$, and $\gamma_t(x_1,x_2)$ is the density of $(X^1_t,X^2_t)$. The proof of \eqref{e:pit} is a minor modification of the proof of \eqref{e:cov} (see \cite{bmr20}).

 With \eqref{e:pit} in hand, the steps of the proof in \cite{bmr20} (i.e.\ the case $m=2$) are:
 \begin{itemize}
     \item Assume that $f$ takes values in a finite-dimensional space of functions, and identify $f$ with a finite-dimensional standard Gaussian vector in this space;
     \item Apply \eqref{e:pit} to the events $A^1 = \{ \cap_{I_i \in \mathcal{C}_1} \mathcal{P}_{I_i}^{\sigma_i} \}$ and $A^2 = \{ \cap_{I_i \in \mathcal{C}_2} \mathcal{P}_{I_i}^{\sigma_i} \}$;
     \item Applying the coarea formula, express the right-hand side of \eqref{e:pit} as the right-hand side of \eqref{e:inter} (in the case $m=2$);
     \item Pass to the limit as the dimension of the function space tends to infinity.
 \end{itemize}

To extend to $m \ge 2$, we only need to generalise \eqref{e:pit} as follows. Let $ (X^1_t,\ldots, X^m_t)$ be a Gaussian vector in $\R^{nm }$ with covariance matrix as in \eqref{e:cov}. Let $A^\ell$, $1 \le \ell \le m$, be domains in $\R^n$ as in \eqref{e:pit}. Then $t \mapsto \mathbb{P}[\cap_{1 \le \ell \le m} \{ X^\ell_t \in A^\ell \} ]$ is absolutely continuous on $[0,1]$ with a.e.\ derivative
\begin{equation}
    \label{e:pit2}
 \frac{d}{dt} \mathbb{P} \left[\cap_{1 \le \ell \le m} \{ X^\ell_t \in A^\ell \} \right]= \sum_{1 \le \ell < \ell' \le m} \int_{\partial A^\ell \times \partial A^{\ell'}}\langle\nu_{A^\ell}(x_\ell),\nu_{A^{\ell'}}(x_{\ell'})\rangle \gamma^{\ell,\ell'}_t(x_\ell,x_{\ell'}) \, d x_\ell dx_{\ell'}, 
 \end{equation}
 where $\nu_{A^i}$ and $\int dx_i$ are as in \eqref{e:pit}, and $\gamma^{\ell,\ell'}_t(x_\ell,x_{\ell'}) $ denotes the density of $(X^\ell_t,X^{\ell'}_t)$.

Assuming that $f$ takes values in a finite-dimensional space of functions, and applying \eqref{e:pit2} to the events $A^\ell = \{ \cap_{I_i \in \mathcal{C}_\ell} \mathcal{P}_{I_i}^{\sigma_i} \}$, the remainder of the proof proceeds as in the case $m=2$.
\end{proof}

\subsection{Proof of the splitting for simple clusterings}
We now show how to apply Proposition \ref{p:inter} to derive the splitting results in Propositions \ref{p:wsplit} and \ref{p:ssplit}. We begin with a useful preliminary lemma:

\begin{lemma}
\label{l:gdb}
Let  $f_t$ be as in Proposition \ref{p:inter}. Then there exists  $c > 0$ such that, for every  $|x-y| \ge 1$ and $t \in [0,1]$,
\[  1/c < \det( \Sigma_{x,y;t} ) < c \]
where $\Sigma_{x,y;t}$ is the covariance matrix of $(f_t(x),f_t(y),f_t'(x),f'_t(y))$. 
\end{lemma} 
\begin{proof}
Let $\mathbf K(x,y)$ be the $2\times 2$ covariance matrix for the Gaussian vectors $(f(x),f'(x))$ and $(f(y),f'(y))$. Then, recalling \eqref{covwitht}, if $x \in I_i \in \mathcal{C}_\ell$ and $y \in I_j \in \mathcal{C}_{\ell'}$ for some $\ell \neq \ell'$ we have
\[ \Sigma_{x,y;t}=\begin{pmatrix}\mathbf{K}(x,x) & t \mathbf{K}(x,y) \\ t\mathbf{K}(x,y) & \mathbf{K}(y,y)\end{pmatrix} \]
and otherwise 
\[ \Sigma_{x,y;t}=\Sigma_{x,y;1} = \begin{pmatrix}\mathbf{K}(x,x) & \mathbf{K}(x,y) \\ \mathbf{K}(x,y) & \mathbf{K}(y,y)\end{pmatrix} .\]
We claim that, in either case, $t \mapsto \textrm{det}( \Sigma_{x,y;t})$ is non-increasing, so that
  \[c=\det(\Sigma_{x,y;0})\geq \textrm{det}( \Sigma_{x,y;t})  \ge  \textrm{det}(\Sigma_{x,y;1} ) =: d(x-y) .  \]
  Indeed, in the latter case this is immediate, and in the former case it follows directly from the proof of Lemma A.5 in \cite{bmr20}. Further, $d(x)$ depends continuously on $t$, $K(x)$, $K'(x)$ and $K''(x)$, and so is a continuous function of $(x,t)$ outside a neighbourhood of $\{x=0\}$. Moreover, recalling Lemma \ref{l:m}, $d(x) \to \textrm{det}( \Sigma)^2$ as  $|x| \to \infty$, where $\Sigma$ is the covariance matrix of $(f(0),f'(0))$. Hence, since $d(x) > 0$ for every $x \neq 0$ by Lemma \ref{l:nondegen}, $d(x)$ is bounded below over $|x| \ge 1$ by compactness.
\end{proof}

We can now complete the proof of the weak splitting:

\begin{proof}[Proof of Proposition \ref{p:wsplit}]
Let $(I_i) \in \mathcal{Q}^k_{r,s}$ and $(\sigma_i) \in \{-,+\}$ be given. By Proposition \ref{p:inter}, and since $\prod_i \P[  \calP_{I_i}^{\sigma_i}  ] \geq P(r)^k$ by stationarity and monotonicity, we have
\[  \Big|  \frac{ \P[ \cap_i \calP^{\sigma_i}_{I_i} ] }{   \prod_i \P[  \calP_{I_i}^{\sigma_i}  ]  }  - 1 \Big| \le \frac{c  k^2 r^2 \bar{K}(s)  E }{ P(r)^k} \]
where 
\[ E :=  \sup_{i < j , x \in I_i, y \in I_j, t \in [0,1] }  p_t(x,y) .  \]
By \eqref{e:persplit}, and since $P(r) = G(r) / 2 < 1$, it remains to show that $E$ is uniformly bounded. 

We shall assume that $x \in \textrm{int}(I_i)$ and $y \in \textrm{int}(I_j)$ since the cases $x \in \partial I_i$ or $y \in \partial I_j$ are similar (and simpler). Note that $\tilde{\varphi}_t(x,y) = (2\pi)^{-2} (\textrm{det}( \Sigma_{x-y;t}) )^{-1/2}$ is uniformly bounded by Lemma \ref{l:gdb}.  By discarding the indicator, and it remains to show that
\[  \E \big[   |f_t''(x)| |f_t''(y)|  \big| f_t(x) = f_t(y) = f_t'(x) = f_t'(y) = 0 \big]   \]
is uniformly bounded. By the Cauchy-Schwarz inequality, the fact that $f_t$ is centred, and Gaussian regression, we have
\begin{align*}
  &\E \big[   |f_t''(x)| |f_t''(y)|  \big| f_t(x) = f_t(y) = f_t'(x) = f_t'(y) = 0 \big] \\
  & \qquad \le \max_{z = x,y}  \E \big[  f_t''(z)^2  \big|  f_t(x) = f_t(y) = f_t'(x) = f_t'(y) = 0  \big]  \\ 
  & \qquad \le \max_{z = x,y } \textrm{Var}[ f_t''(z) ]  =  \textrm{Var}[ f''(0) ]  .
  \end{align*}
  where the final equality is since $f_t$ is equal in law to $f$ when restricted to an interval $I_i$. \end{proof}

To prove the strong splitting we establish a more refined `self-bounding' property of the integrand in \eqref{e:inter}:

\begin{proposition}
\label{p:depin}
Under the assumptions of Proposition \ref{p:ssplit}, there exists a constant $c > 0$ such that: if $r \ge c$ and $s \ge cr$, then for every collection $(I_i)_{i \le k} \in \mathcal{Q}_{r,s}^k$, every $(\sigma_i)_i \in \{-,+\}$, and every $x \in I_i$, $y \in I_j$, $i \neq j$, and $t \in [0,1]$,
\[  | p_t(x,y) |  \le  c \P[A_t] ( \log(1/ \P[A_t] ) )^3 \]
where $p_t(x,y)$ and $A_t$ are defined as in Proposition \ref{p:inter}.
\end{proposition}

\begin{proof}[Proof of Proposition \ref{p:ssplit} assuming Proposition \ref{p:depin}]
Let $(I_i) \in \mathcal{Q}^k_{r,s}$ and $(\sigma_i) \in \{-,+\}$ be given. Let $A_t$ be  defined as in Proposition \ref{p:inter}, and abbreviate $g(t) = \P[A_t]$, so that $g(1) =   \P[ \cap_i \calP^{\sigma_i}_{I_i} ]$ and $g(0) =\prod_i \P[  \calP_{I_i}^{\sigma_i}  ]$. We seek to bound $g(1)/g(0)$.

By Proposition \ref{p:inter}, 
\[   g(1) - g(0) = \int_0^1  g'(t) \, dt \]
where 
\[ g'(t) = \sum_{i <  j}  \sigma_i \sigma_j  \int \int K(x-y)  p_t(x,y) \, \mu_{I_i}(dx) \mu_{I_j}(dy)        .\]
By Proposition \ref{p:depin}, and since $r$ is sufficiently large, we have
\[ |g'(t)| \le  c k^2 r^2 \bar{K}(s) g(t) ( \log(1/g(t)))^3 .  \]
Abbreviating $\gamma = c k^2 r^2 \bar{K}(s)$ we can write the above as
 \[
 | (\log g(t) )' | \le \gamma  ( \log(1/g(t)))^3  .
 \] We introduce $u(t) := -\log g(t)$, and thus 
\[
|u'(t)|  
   \le \gamma\, u(t)^3 ,
\qquad t\in[0,1].
\]
Separating variables gives
\[
-\gamma\le \int_{u_0}^{u_1} \frac{du}{u^3} \le \gamma,
\]
where $u_0=u(0)$ and $u_1=u(1)$.  
Hence
\[
-\gamma\le-\frac{1}{2u_1^2} + \frac{1}{2u_0^2} \le \gamma \implies  \frac 1{\sqrt{\frac1{u_0^2}+2\gamma}}\leq u_1 \le  \frac 1{\sqrt{\frac1{u_0^2}-2\gamma}}.
\]
Taking $r$ and $s$ sufficiently large we can assume that 
$\gamma u_0^2$ is small, and thus we have  
\[
|u_1 - u_0| \le 2\gamma\, u_0^3 .
\]
Since $g(1)/g(0)=\exp(-(u_1-u_0))$, we obtain
\begin{equation}
\label{quot}
1 - 2\gamma \bigl(\log(1/g(0))\bigr)^3
\;\le\;
\frac{g(1)}{g(0)}
\;\le\;
1 + 2\gamma \bigl(\log(1/g(0))\bigr)^3 .
\end{equation}
Together with the fact that $g(0)\geq G(r)^k$, this completes the proof.
\end{proof}

We now give the proof of Proposition \ref{p:depin}, which is inspired by the proof of the similar `de-pinning' estimate \cite[Proposition 3.16]{mm25}.

\begin{proof}[Proof of Proposition \ref{p:depin}]
Again we assume that $x \in \textrm{int}(I_i)$ and $y \in \textrm{int}(I_j)$, $i \neq j$, since the other cases are similar (and simpler). 
As observed in the proof of Proposition \ref{p:wsplit}, $p_t(x,y)$ and $\tilde{\varphi}_t(x,y)$ are uniformly bounded. Hence it remains to prove that
\[ \E \big[   |f_t''(x)| |f_t''(y)|   \id_{A^{xy}_t} \big| f_t(x) = f_t(y) = f_t'(x) = f_t'(y) = 0 \big] \le   c \P[A_t] ( \log(1/ \P[A_t] ))^3 \]
for all $\P[A_t]$ is sufficiently small. 

Decomposing $1 = \id_{f_t''(x) \le u} + \id_{f''_t(x) > u} = \id_{f_t''(y) \le u} + \id_{f''_t(y) > u} $ we can bound, for any $u > 0$,
\begin{align}
\nonumber
&  \E \big[   |f_t''(x)| |f_t''(y)|  \id_{A^{xy}_t}  \big| f_t(x) = f_t(y) = f_t'(x) = f_t'(y) = 0 \big] \\
\label{e:b1} & \qquad \le u^2 \P \big[ A^{xy}_t  \big| f_t(x) = f_t(y) = f_t'(x) = f_t'(y) = 0 \big]  \\
\nonumber & \qquad \qquad + 2u  \max_{z = x,y}   \E \big[   |f_t''(z)| \id_{|f_t''(z)| > u} \big|   f_t(x) = f_t(y) = f_t'(x) = f_t'(y) = 0 \big] \\
\nonumber & \qquad \qquad \qquad + \max_{z = x,y} \E \big[   |f_t''(z)|^2 \id_{|f_t''(z)| > u}  \big|   f_t(x) = f_t(y) = f_t'(x) = f_t'(y) = 0 \big] .
\end{align}
By Gaussian regression, conditionally on $\{f_t(x) = f_t(y) = f_t'(x) = f_t'(y) = 0\}$ the variables $f''_t(z)$, $z = x,y$ are centred Gaussian with variance bounded by $\textrm{Var}(f''(0))$. Hence if $u \ge 1$ the second and third terms on the right-hand side of \eqref{e:b1} are at most  $c_1 e^{- c_2 u^2}$ for constants $c_1,c_2 > 0$. Setting $u  = \sqrt{(1/c_2) \log(1/ \P[A_t])}$, which satisfies $u \ge 1$ as long as $\P[A_t]$ is sufficiently small, these are bounded by $ c_1 \P[A_t]$. The first term on the right-hand side of \eqref{e:b1} is then bounded by
\[ (1/c_2) \log(1/ \P[A_t] ) \P \big[A^{xy}_t  \big| f_t(x) = f_t(y) = f_t'(x) = f_t'(y) = 0 \big]  \]
so it remains to prove that
\begin{equation}
\label{e:b2}
  \P \big[A^{xy}_t  \big| f_t(x) = f_t(y) = f_t'(x) = f_t'(y) = 0 \big] \le c_3   \P [ A_t  ]    (\log ( 1/  \P[A_t] ))^2 . 
  \end{equation}
  
Henceforth we drop $x$, $y$, and $t$ from our notation for simplicity. By Gaussian regression we can write the process $f_t$, conditionally on the vector $ v = (f_t(x) , f_t(y) , f_t'(x) , f_t'(y) )$, as
\[ f_t( \cdot ) \stackrel{d}{=} \mu(\cdot)^T  v + g(\cdot) \] 
where, for each $z \in \R$, $\mu(z)$ is a deterministic vector (written explicitly in \eqref{e:mu} below), and $g$ is a process distributed as $f_t$ conditionally on $ (f_t(x) , f_t(y) , f_t'(x) , f_t'(y) ) = 0$; in particular neither $\mu$ nor $g$ depend on $v$.

\smallskip
We make the following claim, which is the only part of the proof that uses the fact that Condition \ref{c:ss} holds. Recall that $x \in I_i$ and $y \in I_j$ for $i\neq j\in \{1,.., k\}$, and that $\sigma_i, \sigma_j \in \{-,+\}$. Let $H$ denote the RKHS associated to $g$.

\begin{claim}
\label{c:h}
There exist $ c > 0$ and an open set $\mathcal{O} \subset \R^2$, which depend only on $f$ and $k$, and a function $h \in H$, such that:
\begin{enumerate}
\item $h(x) = h(y) = h'(x) = h'(y) = 0$ and $\|h \|_{H} \le c$.
\item For every $w = (w_1,w_2,w_3,w_4)$ such that $(w_1,w_3) \in \sigma_i \mathcal{O}$ and $(w_2,w_4) \in \sigma_j \mathcal{O}$, and every $1 \le m \le k$ and $z \in I_m$,
\begin{equation}
\label{e:c2}
\mu(z)^T w \neq h(z) \qquad \text{and} \qquad \textrm{sgn} \big( \mu(z)^T w  -  h(z) \big) = \sigma_m  . 
\end{equation} 
\end{enumerate}
\end{claim}

Let us conclude the proof assuming this claim. Let $\mathcal{O} \subset \R^2$ be as in the claim and define 
\[ \mathcal{A} = \{ (w_1,w_2,w_3,w_4) : (w_1,w_3) \in \sigma_i \mathcal{O} \text{ and } (w_2,w_4) \in \sigma_j \mathcal{O} \} .  \]
Since $\mathcal{O}$ is open, and $v$ is uniformly non-degenerate by Lemma \ref{l:gdb}, there is a $c_1 > 0$ such that, for every $\delta \in (0,1)$, $\P[v \in \delta \mathcal{A}] \ge c_1 \delta^4$.  For $w \in \R^4$, define 
\[ \Psi(w) := \P[ f_t \in A_t | v = w] = \P[g +  \mu^T w \in A_t] .\]
We claim that, for every $\delta, T > 0$,
\begin{equation}
\label{e:b3}
\inf_{w \in \delta \mathcal{A}} \Psi(w) \ge e^{-c_2/2}  e^{-c_2 \delta T} \Psi(0) - e^{-T^2/2} ,
\end{equation}
where $c_2 > 0$ is the constant from Claim \eqref{c:h}. Assuming this, and since also
\[ \inf_{w \in \delta \mathcal{A}} \Psi(w) \le \P[ A_t | v  \in \delta \mathcal{A} ] \le  \frac{  \P [A_t] }{ \P[ v  \in \delta \mathcal{A}  ] } , \]
we deduce that, for every $T > 0$ and $\delta \in (0,1)$,
\[ \Psi(0) \le e^{c_2/2}  e^{c_2 \delta T} \big( c_1 \delta^{-4} \P[A_t] + e^{-T^2/2}  \big).\]
Setting $T = \sqrt{2} \sqrt{-\log \P[A_t]}$ and $\delta = 1 / \sqrt{-\log \P[A_t]}$, which satisfies $\delta \in (0,1)$ sufficiently small as long as $\P[A_t]$ is sufficiently small, yields 
\[ \Psi(0) \le e^{c_3} \P[A_t] ( c_1 (\log (1/\P[A_t] ) )^2 + 1) \]
which proves \eqref{e:b2} and hence Proposition \ref{p:depin}, subject to \eqref{e:b3} (and Claim \ref{c:h}).

To establish \eqref{e:b3}, let $h$ be as in Claim \ref{c:h}.  By the second property of Claim \ref{c:h}, for every $\delta > 0$ and $w \in \delta \mathcal{A}$,
\[  \{ g  + \delta h  \in A_t  \}  \quad \Longrightarrow \quad   \{g +  \mu^T w  \in A_t\}  . \]
Therefore, for every $\delta > 0$ and $w \in \delta \mathcal{A}$,
\[ \Psi(w) = \P[g +  \mu^T w \in A_t] \ge \P[g + \delta h \in A_t] .\]
Let $\varphi$ be the standard Gaussian density. We make use of the following bound, valid for arbitrary $F : \R \to  [0,1]$ and $T>0, \delta \in (0,1)$ sufficiently small
\begin{align}
\nonumber \E[F(Z+\delta)] & = \int F(z+\delta) \varphi(z) \,dz = \int F(z) \varphi(z-\delta)\, dz = e^{-\delta^2/2}  \int e^{ z \delta}  F(z) \varphi(z) dz \\
\nonumber & \ge e^{-\delta/2} e^{-T\delta}  \int_{-T}^\infty  F(z) \varphi(z)dz  \\
\label{e:z} & \ge  e^{-\delta/2} e^{-T\delta}  \E[F(z)] -     \P[Z \le - T]  .
 \end{align}
We decompose $g \stackrel{d}{=} ( Z / \|h\|_{H} )  h + \tilde{g}$, where $Z$ is a standard Gaussian random variable, and $\tilde{g}$ is independent of $Z$, and write
\[ \P[g + \delta h \in A_t]  =   \P \big[  (( Z + \delta \|h\|_{H} ) / \|h\|_{H} )  h + \tilde{g}   \in A_t \big] = \E_{\tilde g}\E_Z \big[ F_{\tilde{g}}(Z +  \delta \|h \|_{H} ) \big| \tilde{g}  \big]   \]
where $F_{\tilde{g}}$ is the indicator function
 \[ F_{\tilde{g}}(z) := \id \{ ( z/ \|h\|_{H}) h + \tilde{g} \in A_t \} .\]
Recalling $ \|h \|_{H}  \le c_2$, we can choose  $\delta \in (0,1)$ sufficiently small such that $\delta \|h \|_{H}\in (0,1)$, applying \eqref{e:z} gives that
 \[ \P[g + \delta h \in A_t] \ge  e^{-c_2/2 -T\delta c_2}  \P [g \in A_t ] -     \P[Z \le - T]   . \]
By the standard tail bound $\P[Z \le - T]  \le e^{-T^2/2}$, this concludes the proof of \eqref{e:b3}
\end{proof}

\begin{proof}[Proof of Claim \ref{c:h}]
We begin by defining suitable $\mathcal O$ and $h$. First, since $K > 0$ and $|K'(x)| / K(x) < c_1$, we may choose a bounded open set $\mathcal{O} \subset \R^2$ satisfying, for every $(w_1,w_3) \in \mathcal{O}$ and $x \in \R$,
\begin{equation}
\label{e:w}
\frac{ w_1 K(x) }{K(0)}   +  \frac{ w_3 K'(x) }{-K''(0)}  >   K(x)/ 2 .
  \end{equation}
Next, for $1 \le m \le k$, let $d_m \in I_m$ denote the midpoint of the interval $I_m$, and define 
\[ \tilde{w} =   \sum_{m \neq i,j} \sigma_m   \big(  K(x-d_m) , K(y-d_m), K'(x - d_m), K'(y-d_m) \big)  \]
and
  \begin{equation}\label{defineh} h(\cdot) =     \mu(\cdot)^T \tilde{w}  - \sum_{m \neq i,j}  \sigma_m K(\cdot - d_m)     . \end{equation}
\noindent Let us verify the two properties in Claim \ref{c:h}. In the following the constants $c > 0$ may change from line to line but depends only on $f$ and $k$.
 
\textbf{(1).} For any $w = (w_1,w_2,w_3,w_4)$, the function $p(z) = \mu(z)^T w $ satisfies $p(x) = w_1$, $p(y) = w_2$, $p'(x) = w_3$, and $p'(y) = w_4$, since  by definition $p(z) = \E[ f_t | v = w ]$. Hence $h(x) = h(y) = h'(x) = h'(y) = 0$ by construction. To bound $\|h\|_H$ we first observe that 
\begin{align*}
 \| \mu(z)^T \tilde{w} \|_H  = \| \Sigma_{x,y,t}^{-1/2}  \tilde{w} \|_{L^2}   \le \lambda_{\min}^{-1/2} \| \tilde{w} \|_{L^2}     \le c   \lambda_{\min}^{-1/2} \max_{m\neq i, j}\max_{z = x,y} \big\{   |K(z-d_m)| ,  |K'(z - d_m)| \big\}    
  \end{align*}
where recall that $ \Sigma_{x,y,t}$ denotes the covariance matrix of $v = (f_t(x), f_t(y), f'_t(x), f'_t(y))$, and $ \lambda_{\min}$ is its minimum eigenvalue. Since $v$ is uniformly non-degenerate on $|x-y| \ge 1$ and $t \in [0,1]$ by Lemmas \ref{l:gdb} and \ref{l:m}, we have that $\| \mu(z)^T \tilde{w} \|_H < c$. Hence, by the triangle inequality
\[ \| h \|_{H} \le   \| \mu(z)^T \tilde{w} \|_{H} + (k-2) \| K \|_H   < c .  \]
 
\textbf{(2).} By Gaussian regression we may write $\mu(z)$  as
\begin{equation}
\label{e:mu}
\mu(z) = u(z)^T \Sigma^{-1}
\end{equation}
where
\[ u(z) =  \big(  t^{(i)} K(x-z),  t^{(j)} K(y-z),  t^{(i)} K'(x-z),  t^{(j)} K'(y-z) \big ) \] where, for $1 \le m \le k$, we define $t^{(m)} = t$ if $z \in I_{m'}$ for $m \neq m'$, and $ t^{(m)} = 1$ otherwise. 
And the covariance matrix $\Sigma$ is 
\[ \Sigma = \Sigma_{x,y;t} =  \mtrfor{\rowfor{K(0)}{t K(x-y)}{0}{-t K'(x-y)}}{\rowfor{t K(x-y)}{K(0)}{t K'(x-y)}{0}}{\rowfor{0}{t K'(x-y)}{-K''(0)}{-t K''(x-y)}}{\rowfor{-t K'(x-y)}{0}{-t K''(x-y)}{-K''(0)}} . \]

 By Lemma \ref{l:m}, the matrix $\Sigma^{-1}$ converges as $|x - y| \to \infty$ to the diagonal matrix with entries $(1/K(0), 1/K(0), -1/K''(0), -1/K''(0))$. In particular
 \begin{align}
\label{e:mubound} &  \mu(z)^T w   = \frac{w_1  t^{(i)} K(x-z)}{K(0)} +  \frac{w_2  t^{(j)} K(y-z)}{K(0)}  + \frac{ w_3  t^{(i)} K'(x-z)}{-K''(0)} + \frac{w_4   t^{(j)}K'(y-z)}{-K''(0)}  \\
 \nonumber &   \quad + o_{|x-y| \to \infty}(1) \times  \|w\|_\infty   \max\big\{ t^{(i)} |K(x-z)|, t^{(i)} |K'(x-z)| ,t^{(j)} |K(y-z)|, t^{(j)} |K'(y-z)|   \big\}  
 \end{align}
  uniformly over  $w = (w_1,w_2,w_3,w_4)$ and $t \in [0,1]$.
  
Recall that we seek to prove that $ \textrm{sgn} \big( \mu(z)^T w  -  h(z) \big) = \sigma_m$ for all $z \in I_m$, $1 \le m \le k$. There are three cases: $z \in I_i$, $z \in I_j$, and $z \in \cup_{m \neq i,j} I_m$.  Let us consider the first case, with the second case identical. Without loss of generality we may assume that $\sigma_i = +$, so we need to verify $ \mu(z)^T w > h(z)$ for $z \in  I_i$.  By \eqref{defineh} and \eqref{e:mubound}, and recalling that $\sup_x |K'(x)| / K(x) < \infty$, we have 
 \[ h(z) \le c \max_{q \in T} K(q)  \le  c  \bar{K}(s)  \]
 where $T = \cup_{m \neq i,j} \{ x-d_m, y-d_m, z-d_m\}$. Similarly, recalling also \eqref{e:w} and that $\mathcal{O}$ is bounded,
 \begin{align*}
 \mu(z)^T w & \ge    \Big( \frac{w_1 K(x-z)}{K(0)} + \frac{ w_3 K'(x-z)}{-K''(0)} \Big) -  c K(y-z) - o_{|x-y| \to \infty}( 1) K(x-z)   \\
 &  \ge K(x-z)/2  -  c K(y-z) - o_{|x-y| \to \infty}( 1) K(x-z)   \\
 & \ge \min_{|x| \le r} K(x) / 3 -   c \bar{K}(s) 
 \end{align*}
if $|x-y| \ge s$ is sufficiently large. Combining with \eqref{e:decay} we see that $h(z) < \mu(z)^T w$ as long as $r \ge 1$ and $s/r \ge c$ is sufficiently large. 
   
 We turn to the remaining case. Suppose $z \in I_m$, $m \neq i,j$. Without loss of generality assume that $\sigma_m = -$, so we need to verify that $h(z) > \mu(z)^T w$ for $z \in I_m$. Similarly to the previous case,
 \[ h(z) \ge   K(z - d_m) - c  \max_{q \in T'}  K(q)    \ge    \min_{|x| \le r} K(x) - c \bar{K}(s)    \]
where $T' = T \setminus \{z - d_m\}$, and 
\[   \mu(z)^T w  \le   c     \max_{q = x-z,y-z}K(q)  \le c \bar{K}(s) ,   \]
and again we conclude using \eqref{e:decay}.
\end{proof}

\subsection{Decoupling for clustered intervals}
\label{ss:cd}
We now prove weaker decoupling results for clusterings of intervals. Let $\bar{\mathcal{Q}}^k_{r,\bar{r},s}$ denote the set of clusterings $(\mathcal{C}_\ell)_{\ell \ge 1}$, where each cluster is a contiguous block of intervals in the left-to-right ordering, and such that:
\begin{itemize}
   \item The collection $(\mathcal{C}_\ell)_\ell$ contains at most $k$ intervals;
   \item Every interval in $(\mathcal{C}_\ell)_\ell$ has length at most $r$;
    \item For every $\ell \ge 1$, $\cup_{I_i \in \mathcal{C}_\ell} I_i$ is contained in an interval of length at most $\bar{r}$;
    \item For every $\ell \neq \ell'$, and every $I_i \in \mathcal{C}_\ell$ and $I_j \in \mathcal{C}_{\ell'}$, the intervals $I_i$ and $I_j$ are separated by a distance at least $s$.
\end{itemize}
Then we have the following analogues of Propositions \ref{p:wsplit} and \ref{p:ssplit}:

\begin{proposition}[Clustered decoupling, weak version]
\label{p:cwsplit}
Suppose $f$ is as in Proposition \ref{p:wsplit}. Then for every $k \ge 2$ there exists a $c > 0$ such that for every $r,\bar{r},s \ge 1$, and every  $(\mathcal{C}_\ell)_{\ell \ge 1} \in \bar{\mathcal{Q}}^k_{r,\bar{r},s}$,
\[  \P[ \cap_i \mathcal{G}_{I_i} ] \le  \prod_{\ell \ge 1} \P[ \cap_{I_i \in \mathcal{C}_\ell}  \mathcal{G}_{I_i} ] + c  r^2  \bar{K}(s)  .\]
\end{proposition}

\begin{proposition}[Clustered decoupling, strong version]
\label{p:cssplit}
Suppose $f$ is as in Proposition \ref{p:ssplit}. Then for every $k \ge 2$ there exists a $c > 0$ such that for every $ \bar{r} \ge r \ge c$, $s \ge c \bar{r}$, and $(\mathcal{C}_\ell)_{\ell \ge 1} \in \bar{\mathcal{Q}}^k_{r,\bar{r},s}$,
\[  \P[ \cap_i \mathcal{G}_{I_i} ] \le  c \Big(\prod_{\ell \ge 1} \P[ \cap_{I_i \in \mathcal{C}_\ell}  \mathcal{G}_{I_i} ]   \Big)^{1 - c r^2 \bar{K}(s)   } . \]
\end{proposition}

\begin{remark}
In fact, compared to Proposition \ref{p:ssplit}, we may relax the conditions in Proposition \ref{p:cssplit} slightly by replacing \eqref{e:decay} with the weaker assumption that, for some $c_1 > 0$,
\begin{equation}
\label{e:decay2}
\min_{|x| \le r} K(x) \ge \bar{K}(c_1 r) / c_1 .
\end{equation}
\end{remark}

\begin{proof}[Proof of Proposition \ref{p:cwsplit}]
As in the proof of Proposition \ref{p:wsplit}, by Proposition \ref{p:inter} we have 
\[  \Big| \P \big[ \cap_i \calP^{\sigma_i}_{I_i} \big] -  \prod_{\ell \ge 1} \P \big[ \cap_{I_i \in \mathcal{C}_\ell} \calP_{I_i}^{\sigma_i}  \big] \Big| \le c  r^2 \bar{K}(s) E , \]
where 
\[ E := \sup_{\ell < \ell', x \in I_i \in \mathcal{C}_\ell, y \in I_j \in \mathcal{C}_{\ell'}, t \in [0,1]} p_t(x,y) \]
is uniformly bounded, and the result follows the fact that $\mathcal{G}_{I_i}$ is a disjoint union of $\calP^{\pm}_{I_i}$.
\end{proof}

\begin{proof}[Proof of Proposition \ref{p:cssplit}]
This is similar to the proof of Proposition \ref{p:ssplit} but simpler.  First we observe that, since $K > 0$, by Slepian's lemma we have
 \[ \max_{\sigma_i \in\{-,+\} } \P[ \cap_i \calP^{\sigma_i}_{I_i}  ]   = \P[ \cap_i  \calP^{+}_{I_i} ] .  \]
Therefore
 \[  \P[ \cap_i \mathcal{G}_{I_i} ]  \le 2^k \P[ \cap_i  \calP^{+}_{I_i} ]  \qquad \text{and} \qquad \prod_{\ell \ge 1} \P[ \cap_{I_i \in \mathcal{C}_\ell}  \mathcal{P}_{I_i}^+ ] \le \prod_{\ell \ge 1} \P[ \cap_{I_i \in \mathcal{C}_\ell}  \mathcal{G}_{I_i} ]  \]
and so it suffices to prove that
 \[  \P[ \cap_i \mathcal{P}^+_{I_i} ] \le   \Big(  \prod_{\ell \ge 1} \P[ \cap_{I_i \in \mathcal{C}_\ell}  \mathcal{P}_{I_i}^+ ] \Big)^{1 - c r^2 \bar{K}(s)   } .\]
We claim that, similar to Proposition \ref{p:depin}, we have the stronger self-bounding property  
\begin{equation}
    \label{e:sb2}
  | p_t(x,y) |  \le  c \P[A_t]  \log(1/ \P[A_t] )  
  \end{equation}
where $p_t(x,y)$ and $A_t$ are defined as in Proposition \ref{p:inter} for the setting $\sigma_i = +$. Supposing this were true, then abbreviating $g(t) = \P[A_t]$, Proposition \ref{p:inter} implies that
\[ |g'(t)|   \le  c k^2 r^2  \bar{K}(s) g(t) ( \log(1/g(t)) .  \]
 Integrating this over $[0,1]$, and taking $r^2 \bar{K}(s)$ sufficiently small,
\[ g(1) \le  e^{ (\log g(0)) e^{-c r^2 \bar{K}(s)} }  \le  g(0)^{1 -  2 c r^2 K(s)  }  \]
which gives the claimed result after adjusting constants.

It remains to prove \eqref{e:sb2}. Arguing as in Proposition \ref{p:depin}, it suffices to consider $x \in I_i \in \mathcal{C}_\ell$ and $y \in I_j \in \mathcal{C}_{\ell'}$, $\ell \neq \ell'$, and prove (c.f. \eqref{e:b2}) that
\begin{equation}
\label{e:b4}
  \P \big[A^{xy}_t  \big| f_t(x) = f_t(y) = f_t'(x) = f_t'(y) = 0 \big] \le c_3   \P [ A_t  ]  . 
  \end{equation}
 By Gaussian regression we write  $f_t$, conditionally on $ v = (f_t(x) , f_t(y) , f_t'(x) , f_t'(y) )$, as
\[ f_t( \cdot ) \stackrel{d}{=} \mu(\cdot)^T  v + g(\cdot) \] 
where $g$ is a process independent of $v$ and $\mu(z)$ is the deterministic vector defined as follows: for $\ell \ge 1$, define $t^{(\ell)} = t$ if $z \in I_i \in \mathcal{C}_{\ell'}$ for $\ell' \neq \ell$, and $ t^{(m)} = 1$ otherwise. By Gaussian regression we may write $\mu(z)$  as
\begin{equation}
\label{e:mu2}
\mu(z) = u(z)^T \Sigma^{-1}
\end{equation}
where
\[ u(z) =  \big(  t^{(\ell)} K(x-z),  t^{(\ell')} K(y-z),  t^{(\ell)} K'(x-z),  t^{(\ell')} K'(y-z) \big ) \]
and $\Sigma$ is as in \eqref{e:mu}. As in \eqref{e:mubound} we have 
 \begin{align*}
 &  \mu(z)^T w   = \frac{w_1  t^{(\ell)} K(x-z)}{K(0)} +  \frac{w_2  t^{(\ell')} K(y-z)}{K(0)}  + \frac{ w_3  t^{(\ell)} K'(x-z)}{-K''(0)} + \frac{w_4   t^{(\ell')}K'(y-z)}{-K''(0)}  \\
 \nonumber & \quad  + o_{|x-y| \to \infty}(1) \times  \|w\|_\infty    \max  \big\{ t^{(\ell)} |K(x-z)|, t^{(\ell)} |K'(x-z)| ,t^{(\ell')} |K(y-z)|, t^{(\ell')} |K'(y-z)| \big\} 
 \end{align*}
 uniformly over  $w = (w_1,w_2,w_3,w_4)$ and $t \in [0,1]$. 

\smallskip
We claim that if $\mathcal{O} \subset \R^2$ is the set defined Claim \ref{c:h}, then
\begin{equation}
    \label{e:mon}
 \mu(z)^T w > 0  
 \end{equation} 
for every $z \in  \cup_m I_m$ and every
\[ w \in  \mathcal{A} = \{ (w_1,w_2,w_3,w_4) : (w_1,w_3),(w_2,w_4) \in  \mathcal{O} \} .  \]
Indeed suppose $z \in I_m \in \mathcal{C}_\ell$. Then by \eqref{e:w}, and since $|K'(z)| / K(z)$ and $\mathcal{O}$ are bounded,
 \begin{align*}
 \mu(z)^T w & \ge    \Big( \frac{w_1 K(x-z)}{K(0)} + \frac{ w_3 K'(x-z)}{-K''(0)} \Big) -  c K(y-z)  - o_{|x-y| \to \infty}( 1)  \max_{q = x-z,y-z} \{ K(q) \} \\
 &  \ge K(x-z)/2   -  c K(y-z) - o_{|x-y| \to \infty}( 1)  \max_{q = x-z,y-z} \{ K(q) \}  . 
 \end{align*}
If $s / \bar{r}$ is sufficiently large, then by \eqref{e:decay2} we have
\[  \max_{q = x-z,y-z} \{ K(q) \} \le c_1 K(x-z) .\]
If also $|x-y| \ge s$ is sufficiently large,  we see that $\mu(z)^T w > 0$. The cases $z \in I_m \in \mathcal{C}_{\ell'}$ and $z \in I_m \notin  \mathcal{C}_\ell \cup \mathcal{C}_{\ell'}$ are similar, and we omit them.
   
Now for $w \in \R^4$, define 
\[ \Psi(w) := \P[ f_t \in A^{xy}_t | v = w] = \P[g +  \mu^T w \in A_t] . \] 
By \eqref{e:mon} and since the events $\mathcal{P}_{I_i}^+$ are increasing,
\[ \Psi(0) \le \inf_{w \in \mathcal{A}} \Psi(w)   \le \P[ A^{xy}_t | v  \in  \mathcal{A} ] \le  \frac{  \P [A_t] }{ \P[ v  \in  \mathcal{A}  ] }  . \]
Since also $\mathcal{O}$ is open, and $v$ is uniformly non-degenerate by Lemma \ref{l:gdb}, $\P[w \in  \mathcal{A}]$ is bounded below, which yields \eqref{e:b4}. 
\end{proof}

\medskip
\section{Absence of clustering}
\label{s:clus}

\medskip

Let $\mathcal{V}_r$ denote the collection of pairs of compact intervals $(I_1,I_2)$ such that $I_1 \cap I_2$ contains at most one point and $|I_i| \ge r$ for $i = 1,2$. Define the gap \textit{clustering coefficient}
\begin{equation}
\label{e:phi}
\phi_r = \sup_{(I_1,I_2) \in \mathcal{V}_r }  \P[ \calG_{I_1} \cap \calG_{I_2} ] .
\end{equation} 

We aim to bound $\phi_r$ under the conditions of our main result; in particular we wish to show that there exists $\eta > 0$ such that $\phi_r \le G(r)^{1 + \eta + o(1)}$ as $r \to \infty$.

\smallskip
 It is likely that such bounds could be proven by suitably adapting the methods of \cite{ffm25,ffm26} that lead to the bounds on $G(r)$ in Proposition \ref{p:per}. For simplicity, here we take a different approach, first establishing  general bounds on $\phi_r$ in terms of $G(r)$ and related quantities, and then combining these with Proposition \ref{p:per} directly.
 
 \smallskip
 The main results in this section are the following:

\begin{proposition}
\label{p:clus1}
Let $f$ be a continuous SGP and $\eps \in (0,1)$ and $r_0 > 0$ be constants such that, for every $r \ge r_0$
\begin{equation}
\label{e:c1}
 K( \eps r ) \ge 0 \qquad \text{and} \qquad  \inf_{|x| < (1-\eps/2)r }K(x) \ge \bar{K}(r)  .
 \end{equation}
Then $\phi_r \le 2G((2-\eps)r)$ for every $r \ge r_0 / \eps$.
\end{proposition}

Observe that, under assumption (2) of Theorem \ref{t:main}, for every $\eps > 0$ one can find $r_0 = r_0(\eps) > 0$ such that \eqref{e:c1} holds.

\begin{proposition}
\label{p:clus2}
Let $f$ be a $C^1$-smooth SGP satisfying $|K(x)| \le c_0 |x|^{-\alpha}$ for some $c_0 > 0$ and $\alpha > 1$. Then for every $\eps > 0$ there exists a $c_1 > 0$ such that, for sufficiently large $r$ 
\[ \phi_r \le 4 \P\big[ f(x) > -\eps : x \in [0,(1-\eps)r] \big]^2  + e^{-  c_1 r^{(\alpha-1)/2} }  .\]
\end{proposition}

\begin{corollary}\label{c:clus}
Under the conditions of Theorem \ref{t:main} there exists an $\eta > 0$ such that, as $r \to \infty$,
\[ \phi_r \le G(r)^{1+\eta + o(1)} .\] 
\end{corollary}
\begin{proof}
Suppose the conditions in the first item are satisfied, in particular ($\infty$-PD) holds. Then Proposition \ref{p:clus2} implies that, for every $\eps, \alpha > 0$ and sufficiently large $r$,
\[ \phi_r \le 4   \P\big[ f(x) > -\eps : x \in [0,(1-\eps)r] \big]^2  + e^{- c_1 r^{(\alpha-1)/2} }  . \] 
By Proposition \ref{p:per}, for every $\eps' > 0$ we may choose $\eps > 0$ sufficiently small so that
\[  G(r) = e^{-\zeta r + o(r)} \qquad \text{and} \qquad \P\big[ f(x) > -\eps : x \in [0,(1-\eps)r] \big]   \le e^{-(\zeta-\eps')r }. \]
Hence taking $\alpha > 3$
 \[ \phi_r \le  e^{-2(\zeta-\eps')r + o(r) } = e^{- \zeta r (2 - 2 \eps'/\zeta ) + o(r) }  =  G(r)^{2  - 2 \eps'/\zeta  + o(1)} .\]
 Taking $\eps' \to 0$ sufficiently small gives the claim for $\eta = 1$.
 
Suppose the conditions in the second item are satisfied, in particular ($\alpha$-PD) holds. Then  Proposition \ref{p:clus1} implies that, for every $\eps > 0$, 
 \[  \phi_r \le 2G((2-\eps)r) .\]
Combining with $G(r) \sim e^{-\zeta r^{\min\{1,\alpha\}} (\log r)^{\id_{\alpha < 1} } }$  (Proposition \ref{p:per}), and taking $\eps$ sufficiently small gives the claim for $\eta = 2^{\min \{1,\alpha \}}- 1 > 0$.
\end{proof}

Recall the notation $\mathcal{P}^\sigma_I$ from the previous section. Since $\calG_I = \calP_I^+ \cup \calP_I^-$, and by monotonicity, we have
\begin{equation}
\label{e:phi2}
 \phi_r \le 4  \sup_{(I_1,I_2) \in \mathcal{V}'_r, \sigma_1, \sigma_2 \in \{-,+\} } \P[ \calP^{\sigma_1}_{I_1} \cap \calP^{\sigma_2}_{I_2} ]   
\end{equation}
 where $\mathcal{V}'_r$ is the subset of $\mathcal{V}_r$ such that the intervals $I_i$ have length \textit{exactly} $r$, so we focus on bounding the latter quantity. 

\begin{proof}[Proof of Proposition \ref{p:clus1}]
Fix $\eps \in (0,1)$ and $r_0 > 0$ and suppose that \eqref{e:c1} is satisfied. Without loss of generality suppose that $I_1 = [s,t]$ and $I_2 = [u,v]$ with $t-s = r > r_0$, $v-u = r > r_0$, and $t \le u$. Define the interval $I_3 = [u+\eps r, v] \subset I_2$, disjoint from $I_1$.

First observe that, since $K(x) \ge 0$ for $x \ge \eps r$, by monotonicity and Slepian's lemma
 \[ \max_{\sigma_1, \sigma_2 \in \{-,+\} } \P[  \calP^{\sigma_1}_{I_1} \cap \calP^{\sigma_2}_{I_2}  ]  \le \max_{\sigma_1, \sigma_2 \in \{-,+\} } \P[  \calP^{\sigma_1}_{I_1} \cap \calP^{\sigma_2}_{I_3}  ] = \P[  \calP^{+}_{I_1} \cap \calP^{+}_{I_3}  ] . \] 
 
Next consider the auxiliary process
\[ g(x) =  \begin{cases}
f(x) & \text{if } x \in I_1 , \\ f(x   - (u - t) - \eps r )  & \text{if } x \in I_3 . \end{cases} \]
Observe that, for $x \in I_1$ and for $y \in I_3$,
\[  \E[ g(x)g(y) ] = K(y-x - (u-t) - \eps r ) \ge K(y-x) = \E[ f(x)f(y) ]  \]
where the inequality holds for sufficiently large $r$ by the second condition in \eqref{e:c1} once we observe that
\[ \frac{ y-x - (u-t) - \eps r }{y-x} \le 1 - \eps/2 . \]
Therefore, for $x, y\in I_1\cup I_3$, we always have 
\[  \E[ g(x)g(y) ]  \ge  \E[ f(x)f(y) ].  \]
 Applying Slepian's lemma once more
\begin{align*}
& \P[ f \in  \calP^{+}_{I_1} \cap \calP^{+}_{I_3}  ]  \le \P[  g \in \calP^{+}_{I_1} \cap \calP^{+}_{I_3} ] \\
& \qquad =  \P[  f \in  \calP^{+}_{[s,t+ v - (u + \eps r) ]}  ] = G(t-s + v- (u+\eps r) )/2 \le G((2-\eps)r)/2 .
 \end{align*}
Combining with \eqref{e:phi2} we see that $\phi_r \le 4  \times ( G((2-\eps)r)/2 ) = 2 G((2-\eps)r)$ as required.
\end{proof}

\begin{proof}[Proof of Proposition \ref{p:clus2}]
 Let $I_1, I_2, I_3$ be as in the previous proof, and suppose $\sigma_1 = +$ and $\sigma_2 = -$ for simplicity (the proof in the other cases are identical). Let $\delta_r = r^{-(\alpha-1)/4}$. Define $X = f|_{\delta_r \Z} $ and $Y = -f|_{\delta_r \Z}$, let $(Z_i)$ and $(\bar Z)_i$ be i.i.d.\ standard Gaussian vectors indexed by $\eta_r \Z$, and define $\bar{X} = X + \delta_r Z$ and $\bar{Y} = Y + \delta_r \bar{Z}$. 
By monotonicity,
 \[ \P[  \calP^{\sigma_1}_{I_1} \cap \calP^{\sigma_2}_{I_2}  ]   \le \P[ X_i > 0  \ \forall i \in  I_1, Y_j > 0 \ \forall j \in I_3]  . \]
 By standard bounds for i.i.d.\ Gaussian variables, for sufficiently large $r$
 \[ \P \Big[  \max_{i \in \cup \{I_1, I_3\} } \max\{ |Z_i|, |\bar{Z}_i| \} > \eps/\delta_r \Big] \le 2r \delta_r^{-1} e^{- (\eps/\delta_r)^2/2 } \le e^{- c_1 r^{(\alpha-1)/2}  } .\]
 Therefore
 \[ \P[ X_i > 0  \ \forall i \in  I_1, Y_j >0 \ \forall j \in I_3]  \le \P[ \bar{X}_i > -\eps  \ \forall i \in  I_1, \bar{Y}_j >-  \eps \ \forall j \in I_3]  +  e^{- c_1 r^{(\alpha-1)/2}  } .\]
 Let $\rho$ denote the maximum correlation coefficient between $\bar{X}$ and $\bar{Y}$ defined in \eqref{e:mcc}, namely 
 \[  \rho  = \sup_{ \alpha,\beta}   \frac{ | \textrm{Cov}[  \langle \alpha , \bar{X} \rangle , \langle \beta , \bar{Y} \rangle ] | }{ \sqrt{ \textrm{Var}[\langle \alpha , \bar{X}  \rangle ] \textrm{Var}[ \langle \beta , \bar{Y}  \rangle] } } . \]
By Cauchy-Schwarz we have
 \[ \rho \le  \frac{ \max_{i,j} |\textrm{Cov}[\bar{X}_i, \bar{Y}_j] | \|\alpha\|_1 \|\beta\|_1 }{ \delta_r^2  \|\alpha\|_2 \|\beta\|_2  }  \le \bar{K}(\eps r) \delta_r^{-2} (r \delta^{-1}_r)   \le c_2 r^{-\alpha + 1 +3(\alpha-1)/4 } = c_2 r^{-(\alpha-1)/4 }.\]
 Applying the sprinkled decoupling inequality (Proposition \ref{p:sdi}, and suppose for simplicity that $f$ has unit variance) and by monotonicity
 \begin{align*}
  & \P[ \bar{X}_i > -\eps  \ \forall i \in  I_1, \bar{Y}_j >-  \eps \ \forall j \in I_3]  \\
  & \quad \le  \P[ \bar{X}_i > -\eps  \ \forall i \in  I_1] \P[  \bar{Y}_j > - 2 \eps \ \forall j \in I_3] + e^{- \eps^2 / ( 8 (1 + \delta_r^2)  \rho^2 ) } \\
  & \quad \le   \Big( \P[ X_i > -2\eps  \ \forall i \in  I_1]  \!+\!   e^{- c_1 r^{(\alpha-1)/2} } \Big)\! \Big( \P[  Y_j > - 3 \eps \ \forall j \in I_3]  \! + \!  e^{- c_1 r^{(\alpha-1)/2}   }  \Big) \! +  e^{- c_3 r^{(\alpha-1)/2}  }\\
  & \quad \le    \P[ X_i > -3\eps  \ \forall i \in  I_1]  \P[  Y_j > - 3 \eps \ \forall j \in I_3]  + 4e^{- c_3 r^{(\alpha-1)/2}  } .
  \end{align*}  
 Finally by the derivative bound (Lemma \ref{l:derbound})
 \[ \P \Big[  \max_{x \in \cup \{I_1, I_3\} } |f'(x)| > \eps/\delta_r \Big] \le e^{- c_4 \delta_r^{-2} } = e^{- c_4  r^{(\alpha-1)/2}  } , \]
 which implies that
 \[ \P[ X_i > -3\eps  \ \forall i \in  I_1]  \le \P[ f(x) > -4\eps  \ \forall x \in  I_1] +  e^{- c_4  r^{(\alpha-1)/2}  } \] 
 and similarly for $\P[  Y_j > - 3 \eps \ \forall j \in I_3]$. Therefore we have
 \[ \P[  \calP^{\sigma_1}_{I_1} \cap \calP^{\sigma_2}_{I_2}  ]   \le   \P[ f(x) > -4\eps  \ \forall x \in  I_1] \P[ f(x) > -4\eps  \ \forall x \in  I_3]  +  e^{- c_5  r^{(\alpha-1)/2}  } .  \]
Combining with \eqref{e:phi2}, and by monotonicity and adjusting constants, we conclude the proof.
\end{proof}

\medskip
\section{Poisson approximation}
\label{s:pois}
In this section we establish the Poisson approximation in Theorem \ref{t:main}. Recall the scaling function $\theta(r)$ in \eqref{thetafunction}, and the rescaled point process 
\[ \Psi_R= \sum_{z_i \in [0, R]} \delta_{z_i/R,\,\, \theta(z_{i+1}-z_i)-\log R} , \]
where $(z_i)_{i \in \Z} = \mathcal{Z}$ is an arbitrary increasing ordering of the zeros. We prove that $\Psi_R$ converges vaguely to the Poisson point process with intensity  $dx \otimes e^{-y} dy$ on the compactified space $[0,1] \times (\R \cup \infty)$.

\smallskip
We begin by reducing the convergence of the point process to a statement about gap probabilities, following a moment method inspired by \cite{fw25}. This step does not use specific properties of the zero set. We then use the splitting and clustering properties established in the previous sections to complete the proof.

\subsection{Method of moments}
Fix bounded intervals $I \subset \R$ and $A \subset \R \cup \infty$; in particular $A = [a,\infty)$ is permitted. We define $i_1 < i_2$ such that $\bar{I} = [i_1,i_2]$, and $a < b \le \infty$  such that $\bar{A} = [a,b]$. Let $\theta^{-1}$ denote the left-continuous inverse of $\theta$, which satisfies $\theta(\theta^{-1}(s)) = s$, and define $t_R(s) = \theta^{-1}(s + \log R)$, which satisfies $t_R(s) \to \infty$ as $R \to \infty$ and $\theta(t_R(s)) - \log R = s$. Introduce a microscopic scale $\eps_R = R^{-\gamma}$, with $\gamma>0$ to be chosen later. Recall the gap events~$\mathcal{G}$. Then for $x \in \R$ and $0 < y_1 < y_2 \le \infty$ we introduce the non-negative random variable
\begin{align*}
  h_R(x; y_1,y_2) &=  \eps_R^{-1} \Big( \id_{ \mathcal{G}_{[x,x+ y_1 ]} } - \id_{ \mathcal{G}_{[x-\epsilon_R,x+y_1 ]} } -  \id_{ \mathcal{G}_{[x,x+y_2-\eps_R]} } +  \id_{ \mathcal{G}_{[x-\eps_R,x+y_2-\eps_R] } } \Big)  \\
& = \eps_R^{-1}  \id_{ \exists z_i \in [x-\eps_R,x) }\id_{ \mathcal{G}_{[x,x+y_1]} }  \id_{  z_{i+1} \in (x+y_1,x+y_2-\eps_R] } \\& = \eps_R^{-1}  \id_{ \exists z_i \in [x-\eps_R,x) }\id_{z_{i+1}-z_i\in (y_1,y_2] } .
\end{align*}
Roughly speaking, $h_R(x; y_1,y_2)$ is an $\eps_R$-smooth approximation of the indicator that there is a gap starting at the point $x$ with length in $[y_1,y_2]$. For $k \in \mathbb{N}$ and $x_1,\ldots,x_k \in \R$, we define an  `off-diagonal' indicator of distinct gaps
\[ D^{\neq}(x_1,.., x_k):=\id_{ (x_i) \text{ lie in pairwise distinct }[z_{j_i},z_{j_i+1}]} . \]
Next we define
\[ h^+_R(x) = h_R \big(x; t_R(a) -\eps_R ,t_R(b) +\eps_R \big)  \qquad \text{and} \qquad h^-_R(x) =   h_R \big(x; t_R(a) +\eps_R ,t_R(b) -\eps_R \big) \]
with the convention that $t_R(\infty) = \infty$. Finally we define, for $k \in \mathbb{N}$ and $x_1,\ldots,x_k \in \R$, 
\[ m_R^{+,k} =  \int_{[Ri_1-\eps_R, Ri_2+\eps_R]^k} \E \Big[ \prod_{i=1}^k h^+_R(x_i) D^{\neq}(x_1,.., x_k) \Big] \, dx_1 \cdots dx_k \]
and
\[ m_R^{-,k} := \int_{[R i_1+\eps_R,R i_2-\eps_R]^k} \E\Big[ \prod_{i=1}^k h^-_R(x_i) D^{\neq}(x_1,.., x_k)  \Big] \, dx_1 \cdots dx_k . \]
Note that $m_R^{\pm,k}$ depend on the intervals $I$ and $A$, but we have dropped these from the notation.

\smallskip
We reduce the Poisson approximation to the following statement:

\begin{proposition}
\label{p:sc}
Under the assumptions in Theorem \ref{t:main}, for every $k \in \mathbb{N}$ there exists a choice of $\gamma > 0$ such that the following hold:
\begin{enumerate}
    \item For all bounded intervals $I \subseteq [0,1]$ and $A \subset \R$,
\[\lim_{R \to \infty} m_R^{\pm,k} =  \Big( |I| \int_A e^{-y}dy \Big)^k. \]
\item For $I = [0,1]$ and every $A = [a,\infty)$,
\[\lim_{R \to \infty} m_R^{+,1} =   \int_A e^{-y}dy .\]
\end{enumerate}
\end{proposition}

\begin{proof}[Proof of Theorem \ref{t:main} assuming Proposition \ref{p:sc}]
For $x \in \R$ let $(x)_k = \prod_{j = 0}^{k-1} (x-j)$ be its falling factorial, and for a random variable $X$ let $\hat{m}_k(X) := \E \left[ (X)_k \right]$ be its $k^\text{th}$ factorial moment. Fix bounded intervals $I \subset \R$ and $A \subset \R \cup \infty$, $k \in \mathbb{N}$, a suitable choice of $\gamma > 0$ as in Proposition \ref{p:sc}, and define $a < b \le \infty$ such that $[a,b] = \bar{A}$. Then we claim that, as long as $\eps_R < t_R(a)/2$,
\begin{equation}
\label{e:comp}
 m_R^{-,k} \le \hat{m}_k \big(\Psi_R(I\times A) \big) \le m_R^{+,k} . 
\end{equation}
Let us complete the proof assuming \eqref{e:comp}. Suppose that $b < \infty$. Since $\eps_R \le t_R(a)/2$ for sufficiently large $R$, we deduce from \eqref{e:comp} and the first item of Proposition \ref{p:sc} that
\begin{equation}
    \label{e:mm}
\lim_{R\to\infty} \hat{m}_k \big(\Psi_R(I\times A) \big) \to \Big(|I|\int_A e^{-y}dy\Big)^k.
\end{equation}
Since the right-hand side of \eqref{e:mm} is the $k^\text{th}$ factorial moment of a Poisson distribution with mean $|I|\int_A e^{-y}dy$, by a standard criterion (see Proposition 2.1 in \cite{BB13}) this implies the vague convergence of $\Psi_R$ to a Poisson point process with intensity $dx \otimes e^{-y} dy$ on $[0,1] \times \R$. Moreover, by the second item of Proposition \ref{p:sc}, for every $a \in \R$,
\[  \limsup_{R \to \infty} \E \big[\Psi_R \big([0,1] \times [a,\infty) \big) \big] \le \lim_{R \to \infty}  m_R^{+,1} = \int_a^\infty e^{-y} dy  \]
and hence, for every $\varepsilon > 0$,
\[ \lim_{a \to \infty} \limsup_{R \to \infty} \P \big[\Psi_R \big([0,1] \times [a,\infty) \big) > \varepsilon \big]  = 0 . \]
Combining with the vague convergence, this implies that the convergence in fact holds weakly on $[0,1] \times [a,\infty)$ for every $a \in \R$ (\cite[Theorem 4.9]{kall}), and hence vaguely on the compactified space $[0,1] \times (\R \cup \infty)$.

It remains to prove \eqref{e:comp}. To simplify the presentation, suppose that $I = [i_1,i_2]$ and $A = [a,b]$ are closed (the other cases are analogous). Define 
\[ F:= [t_R(a),\, t_R(b)],\quad F_{+} := (t_R(a)-\varepsilon_R,\, t_R(b)+\varepsilon_R],\quad F_{-} := (t_R(a)+\varepsilon_R,\, t_R(b)-\varepsilon_R]. \]
Then by definition, 
\[ h_R^{\pm}(x)
= \varepsilon_R^{-1}\,
\id\big\{ \exists z_i \in [x-\eps_R,x),  z_{i+1}-z_i\in F_{\pm}\big\}. 
\] 
Given a zero configuration $(z_j)$ we have  
\begin{align}
\nonumber \prod_{i=1}^k h_R^+(x_i)\, D^{\neq}(x_1,\dots,x_k)
= &
\varepsilon_R^{-k}\prod_{i=1}^k
\id\Big\{
\exists \ \text{distinct } z_{j_i}:
z_{j_i}\in [x_i-\varepsilon_R,x_i),\ \
z_{j_i+1}-z_{j_i}\in F_{+}
\Big\}\\
\nonumber =&\eps_R^{-k} \sum_{\substack{j_1,\dots,j_k\\ \text{distinct}}} 
\prod_{i=1}^k \id_{[z_{j_i},z_{j_i}+\eps_R)}(x_i)\,
\id\{z_{j_i+1}-z_{j_i}\in F_+\} \\
 \label{e:g+} \geq & \eps_R^{-k} \sum_{\substack{j_1,\dots,j_k\\ \text{distinct}}} 
\prod_{i=1}^k \id_{[z_{j_i},z_{j_i}+\eps_R)}(x_i)\,
\id\{z_{j_i+1}-z_{j_i}\in F\},
\end{align}
where in the last step we used \( F \subset F_+\). For the second equality, note that the sum can only take values in $\{0, \eps_R^{-k}\}$, since for fixed $x_1,\dots,x_k$, at most one tuple of distinct indices $(j_1,\dots,j_k)$ can satisfy $x_i \in [z_{j_i}, z_{j_i}+\eps_R)$  and $z_{j_i+1}-z_{j_i}\in F_+$ for all $i$.

Moreover, by definition
\[ \Psi_R(I\times A)
= \sum_j \id\{z_j\in[Ri_1,Ri_2], \ z_{j+1}-z_{j}\in F \}. \]
Hence its falling factorial satisfies
\begin{align*}
& (\Psi_R(I\times A))_k = \sum_{\substack{j_1,\dots,j_k\\ \text{distinct}}}
\prod_{i=1}^k \id\{z_{j_i}\in[Ri_1,Ri_2],\ z_{j_i+1}-z_{j_i}\in F \}\\
 & \qquad \quad \le \int_{[Ri_1-\eps_R,Ri_2+\eps_R]^k}
\eps_R^{-k} \sum_{\substack{j_1,\dots,j_k\\ \text{distinct}}} 
\prod_{i=1}^k \id_{[z_{j_i},z_{j_i}+\eps_R)}(x_i)\,
\id\{z_{j_i+1}-z_{j_i}\in F\} \,dx_1\cdots dx_k,
\end{align*}
where we used the inequality \(\id_{z\in[s,t]}\leq \int_{s-\epsilon}^{t+\eps}\eps^{-1}\id_{[z,z+\eps)}(x)\,dx\) for any fixed $\eps>0$ and $z\in\mathbb R$. Taking expectations, and comparing with \eqref{e:g+}, 
\[\begin{aligned}
& \hat{m}_k \big(\Psi_R(I\times A) \big)  \\
& \qquad \le   \E\int_{[Ri_1-\eps_R,Ri_2+\eps_R]^k}
\eps_R^{-k} \sum_{\substack{j_1,\dots,j_k\\ \text{distinct}}} 
\prod_{i=1}^k \id_{[z_{j_i},z_{j_i}+\eps_R)}(x_i)\,
\id\{z_{j_i+1}-z_{j_i}\in F\} \,dx_1\cdots dx_k  \\
 & \qquad \le   \int_{[Ri_1-\eps_R, Ri_2+\eps_R]^k} \E \Big[\prod_{i=1}^k h^+_R(x_i) D^{\neq}(x_1,.., x_k) \Big]  \, dx_1 \cdots dx_k  =: m_R^{+,k}.\end{aligned} \]
Analogously, because $F_-\subset F$, one obtains the pointwise bound
\[\eps_R^{-k} \sum_{\substack{j_1,\dots,j_k\\ \text{distinct}}} 
\prod_{i=1}^k \id_{[z_{j_i},z_{j_i}+\eps_R)}(x_i)\,
\id\{z_{j_i+1}-z_{j_i}\in F\}\geq 
\prod_{i=1}^k h_R^-(x_i)\, D^{\neq}(x_1,\dots,x_k), 
\]
which together with the inequality 
 \(\id_{z\in[s,t]}\geq \int_{s+\epsilon}^{t-\eps}\eps^{-1}\id_{[z,z+\eps)}(x)\,dx\), gives that 
\begin{equation*}
 \hat{m}_k \big(\Psi_R(I\times A) \big) \ge m_R^{-,k}. \qedhere
\end{equation*}
\end{proof}

 \subsection{Proof of Proposition \ref{p:sc}}
We need to show that, for $I \subseteq [0,1]$ and $A \subset \R$, 
 \begin{equation}
     \label{e:sc2}
\lim_{R \to \infty} m_R^{\pm,k} =  \Big(|I| \int_A e^{-y} dy \Big)^k ,
 \end{equation}
and also that, for  $I = [0,1]$ and $A = [a,\infty)$,
 \begin{equation}
     \label{e:sc2b}
     \lim_{R \to \infty} m_R^{+,1} =   \int_A e^{-y}dy .
     \end{equation}
     
We collect some properties of the function $G(r)$ that will be essential in the proof:

\begin{claim}
\label{c:gap}
Under the conditions of Theorem \ref{t:main}, for every choice of $\gamma > 0$, and for every $s \in \R$, the following hold: 
\begin{enumerate}
    \item  As $R \to \infty$, uniformly over $x \in [ t_R(s)- \eps_R, t_R(s) + \eps_R]$, 
    \[ G(x) \sim    R^{-1+o(1)} \qquad \text{and} \qquad  G''(x) \le    R^{-1+o(1)} , \]
\item As $R \to \infty$,
\[ t_R(s) = R^{o(1)}  . \]
\end{enumerate}
\end{claim}
\begin{proof}
First we observe that 
\begin{equation}
    \label{e:gbounds}
 G(t_R(s)) \sim    R^{-1+o(1)} \qquad \text{and} \qquad  G''(t_R(s)) \le  G(t_R(s))^{1+o(1)} \sim   R^{-1+o(1)}  . 
 \end{equation}  
Indeed recall that Proposition \ref{p:gpp} states that 
\begin{equation}
    \label{e:gcomp}
-G'(r) \sim G(r)^{1+o(r)} \quad \text{and} \quad G''(r) \le G(r)^{1+o(1)}   .
\end{equation} 
and so \eqref{e:gbounds} follows by combining \eqref{e:gcomp} with the definition 
\[   - G'(t_R(s))  =   e^{- \theta(  \theta^{-1}(s + \log R) ) } = R^{-1} e^{-s}  . \]
 Next we fix $s' < s < s''$, and claim that, for sufficiently large $R$,
\begin{equation}
    \label{e:tcomp}
 t_R(s') \le t_R(s) -  \eps_R \le t_R(s) +  \eps _R \le t_R(s'') 
 \end{equation} 
holds. Indeed by the inverse function rule, for every $u \in \R$,
\[ t'_R(u) = (\theta^{-1})'(u + \log R) = \frac{ - G'(\theta^{-1}(u + \log R))}{G''(\theta^{-1}(u + \log R))} = \frac{ - G'(t_R(u))}{G''(t_R(u))} .\]
Since $-G'$ is non-increasing and $t_R(u)$ is non-decreasing, for $u \in [s',s'']$ we further have
\[ - G'(t_R(u)) \ge -G'(t_R(s'')) = e^{-s''}/R,\]
and by invoking \eqref{e:gcomp} and monotonicity again, also
\[  G''(t_R(u)) \le  - G'(t_R(u))^{1+o(1)} \le  -G'(t_R(s'))^{1+o(1)} = R^{-1+o(1)} . \]
Combining we see that $t'_R(u)  \ge R^{o(1)}$ uniformly over $u \in [s',s'']$, and therefore,  since $\eps_R R^{o(1)} \to 0$, \eqref{e:tcomp} follows from a Taylor expansion. We deduce the first statement of the claim by combining \eqref{e:gbounds}, \eqref{e:tcomp}, and the monotonicity of $G$.

The second statement comes from observing that, under the conditions of Theorem \ref{t:main}, $\theta(r) / \log r \to \infty$, so we have that $\theta^{-1}(r) = e^{o(r)} $, and hence $t_R(s) = R^{o(1)}$. 
\end{proof}

Let us now prove \eqref{e:sc2b}, which is instructive as a warm up for \eqref{e:sc2}. By stationarity, $\mathbb{E}[h^+_R(x)]$ is independent of $x$. Hence for $I= [0,1]$ and $A = [a,\infty)$ we have
\begin{equation}
    \label{e:sc2ba}
m_R^{+,1} =  \E\int_{[-\eps_R, R+\eps_R]} h^+_R(x) \, dx = \mathbb{E}[h_R^+(0) ] (R+2\eps_R)  . 
\end{equation}
Moreover, again by stationarity,
\[ \mathbb{E}[h^+_R(0)] =\eps_R^{-1} \big( G(t_R(a)-\eps_R ) -G(t_R(a))   \big) . \]
Taylor expanding $G$, we have 
\[ \Big| \mathbb{E}[h^+_R(x)]  + G'(t_R(a) ) \Big| \le  (\eps_R/2)  \max_{ s \in [t_R(a) - \eps_R, t_R(a) ]  } |G''(s)| . \]
Recall that by definition $ - G'(t_R(a)) = R^{-1} e^{-a}$. Combining with the first item in Claim \ref{c:gap}, 
\[ \Big| \mathbb{E}[h^+_R(x)]  - R^{-1}  e^{-a}  \Big| = R^{-1-\gamma + o(1)} . \]
Inputting this into \eqref{e:sc2ba} gives \eqref{e:sc2b}.

 \smallskip
 We move onto the proof of \eqref{e:sc2}. In fact we will focus on the statement for $m_R^{+,k}$, since the statement for $m_R^{-,k}$ can be proven analogously. Moreover, we only consider the case $k\geq 2$, since the case $k = 1$ follows from the same reasoning as in the proof of \eqref{e:sc2b} (see \eqref{k11}–\eqref{k13} below). 
 
 Define a mesoscopic scale $s_R = R^\delta$, with $\delta\in(0,1)$ to be chosen later. We introduce a mesoscopic `off-diagonal' indicator
\[  D^{\neq}_R(x_1,\ldots,x_k)=   \id_{ | x_i - x_j | \ge s_R \,  \ \forall i \neq j } .\]
The proof of \eqref{e:sc2} (for $m_R^{+,k}$), and hence Proposition \ref{p:sc}, results from combining the following three statements: for every integer $k \ge 2$, there exists a suitable choice of $\gamma > 0$ and $\delta \in (0,1)$ such that, as $R \to \infty$
\begin{equation}
\label{e:sc3}
 \int_{[Ri_1-\eps_R, Ri_2+\eps_R]^k} \mathbb E \Big[ \prod_{i \leq k} h^+_R(x_i)   \Big] D^{\neq}_R(x_1,\ldots,x_k) \, dx_1 \cdots dx_k   \to  \Big( |I| \int_A e^{-y} dy \Big)^k , 
 \end{equation} 
 and
  \begin{equation}
  \label{e:sc4}
  \int_{[-2R, 2R]^k} \mathbb E \Big[ \prod_{i \leq k} h^+_R(x_i)  D^{\neq}(x_1,\ldots,x_k) \Big]  (1 - D_R^{\neq}) (x_1,\ldots,x_k)   \, dx_1 \cdots dx_k    \to 0,
  \end{equation}
  and the following identity holds
  \begin{equation}
\label{e:sc5}
\int_{[Ri_1-\eps_R, Ri_2+\eps_R]^k} \mathbb E \Big[ \prod_{i \leq k} h^+_R(x_i) (D^{\neq} - 1)(x_1,\ldots,x_k) \Big]  D^{\neq}_R(x_1,\ldots,x_k)  \, dx_1 \cdots dx_k    = 0.
  \end{equation}
Indeed \eqref{e:sc2} follows from \eqref{e:sc3}--\eqref{e:sc5} by decomposing $ D^{\neq}$ in the definition of $m_R^{+,k}$ as
  \begin{equation}
      \label{e:decomp}
   D^{\neq} = D_R^{\neq}  + (1 - D^{\neq}_R) D^{\neq}  + D_R^{\neq}( D^{\neq} - 1),  \end{equation}
  and observing that, by the non-negativity of the integrand 
  \[ \Big( \prod_{i \leq k} h^+_R(x_i)  D^{\neq}(x_1,\ldots,x_k) \Big)  (1 - D_R^{\neq}) (x_1,\ldots,x_k), \]
  the convergence of the integral in $\eqref{e:sc4}$ over the larger domain $[-2R, 2R]^k$ implies its convergence over the smaller domain $[Ri_1-\eps_R, Ri_2+\eps_R]^k$ as well. 

\smallskip
 The proof of \eqref{e:sc5} is straightforward. Observe that if both
 \[ D^{\neq}(x_1,\ldots,x_k) = 0 \qquad \text{and} \qquad  D^{\neq}_R(x_1,\ldots,x_k) = 1, \]
then there exist two points, say $x_1, x_2$, such that $|x_1 - x_2 | > s_R = R^\delta$ and both lie in the same gap $[z_1, z_2]$. By the second item of Claim \ref{c:gap}, this implies that $h_R^+(x_1)=h_R^+(x_2)=0$ by definition. Hence the integrand 
\[ \Big( \prod_{i \leq k} h^+_R(x_i) (D^{\neq} - 1)(x_1,\ldots,x_k) \Big)  D^{\neq}_R(x_1,\ldots,x_k) \]
is identically zero, which gives \eqref{e:sc5}. 

\smallskip
\noindent It remains to establish \eqref{e:sc3} and \eqref{e:sc4}. The proof of the former relies on the splitting bounds from Section \ref{s:split}. The clustered decoupling from Section \ref{ss:cd} and the absence of clustering from Section \ref{s:clus} are used in the proof of the latter.

\subsubsection{Proof of \eqref{e:sc3}.}
We decompose the left-hand side of \eqref{e:sc3} as $A+B$ where
\[
A =  \int_{[Ri_1-\eps_R, Ri_2+\eps_R]^k} \prod_{i \leq k} \mathbb{E}  \Big[  h^+_R(x_i)   \Big] D^{\neq}_R(x_1,\ldots,x_k) \, dx_1 \cdots dx_k  , \]
and 
\[ B = \int_{[Ri_1-\eps_R, Ri_2+\eps_R]^k} \Big( \mathbb{E}  \Big[ \prod_{i \leq k}   h^+_R(x_i)   \Big] - \prod_{i \leq k}   \mathbb{E}  \Big[  h^+_R(x_i)   \Big] \Big)  D^{\neq}_R(x_1,\ldots,x_k) \, dx_1 \cdots dx_k .
\]
Let us focus first on $A$. Arguing as in the proof of \eqref{e:sc2b}, by stationarity
\begin{equation}\label{k11}\mathbb{E}[h^+_R(x)] =\eps_R^{-1} \big( G(t_R(a)-\eps_R ) -G(t_R(a))  - G(t_R(b)) + G(t_R(b) +\eps_R) \big) , \end{equation}
and Taylor expanding $G$
\begin{equation}\label{k12}\Big| \mathbb{E}[h^+_R(x)] -\Big(-G'(t_R(a))+G'(t_R(b)) \Big) \Big|  \le   2 \eps_R \max_{ s \in \{a,b\} } \max_{x \in [t_R(s) - \eps_R, t_R(s) + \eps_R]  } |G''(x)| .\end{equation}
Since $ - G'(t_R(s)) = R^{-1} e^{-s}$,  and by the first item in Claim \ref{c:gap}, 
\begin{equation}\label{k13} \Big| \mathbb{E}[h^+_R(x)]  - R^{-1} \big( e^{-a} - e^{-b}  \big) \Big| = R^{-1-\gamma + o(1)} . \end{equation}
Since
\[ \int_{[Ri_1-\eps_R, Ri_2+\eps_R]^k} \prod_{i \leq k} D^{\neq}_R(x_1,\ldots,x_k) \, dx_1 \cdots dx_k = |I|^k R^k + O(R^{k-1+\delta}) , \]
we conclude that
\begin{align*}
 A  &= \left( R^{-1}(e^{-a} - e^{-b}) +  R^{-1-\gamma+o(1)} \right)^k \left( |I|^k R^k + O(R^{k-1+\delta}) \right)\\
 & = |I|^k (e^{-a} - e^{-b})^k +   O(R^{-\gamma+o(1)} + R^{-1+\delta})\\
 & \to |I|^k (e^{-a} - e^{-b})^k 
\end{align*}
since $\gamma > 0$ and $\delta \in (0,1)$. 

We now show that $B \to 0$, which makes use of the splitting property. Each $h^+_R(x_i)$ is a linear combination of four indicator functions, and so $\prod_{i=1}^k h^+_R(x_i)$ expands into $4^k$ terms of the form
\[ \varepsilon_R^{-k} c_j \prod_{i=1}^k \id_{\mathcal{G}_{I_i}},\]
where $c_j \in \{-1, 1\}$ and $I_{i}$ are intervals of length in $[t_R(a) - \eps_R, t_R(b)+\eps_R]$. In particular, by the second item of Claim \ref{c:gap} the lengths are at most $R^{o(1)}$. Moreover since we assume $D_R^{\neq}(x_1, \dots, x_k) = 1$, the intervals $I_{i,\sigma}$ are separated by at least $s_R - (t_R(b)+\eps_R) \ge s_R - R^{o(1)} \sim R^{\delta}$. Therefore, recalling the definition of $\rho^k_{r,s}$ in Section \ref{s:split}, for each $1 \le j \le 4^k$, we have
\begin{equation}
    \label{e:t} \Big| \mathbb{P}\ \Big( \bigcap_{i=1}^k \mathcal{G}_{I_{i}} \Big) - \prod_{i=1}^k \mathbb{P}(\mathcal{G}_{I_{i}}) \Big| \le \rho^k_{r,s} \prod_{i=1}^k \mathbb{P}(\mathcal{G}_{I_{i}})
\end{equation}  
for $r  = R^{o(1)}$ and $s \sim R^\delta$.  By monotonicity and the first item of Claim \ref{c:gap}, we further have
\begin{equation}
    \label{e:pbound}
 R^{-1+o(1)} = G(t_R(b)+\eps_R)  \le \mathbb{P}(\mathcal{G}_{I_{i}})\le G(t_R(a)-\eps_R) =  R^{-1+o(1)} ,
\end{equation}
and so the right-hand side of \eqref{e:t} is at most $ \rho^k_{r,s} R^{-k+o(1)}$. Summing over the $4^k$ terms and integrating, we obtain
\[ |B| \le  4^k  (4R)^k  \varepsilon_R^{-k} \rho^k_{r,s} R^{-k+o(1)} = R^{k \gamma + o(1)} \rho^k_{r,s} . \]
To finish, suppose that the first assumption in Theorem \ref{t:main} holds, and in particular $K(x)x^\alpha\to0$ for every $\alpha > 0$. By the weak splitting in Proposition \ref{p:wsplit} and \eqref{e:pbound}, we have 
\[ \rho_{r,s}^k\leq R^{o(1)} \bar K(R^{\delta+o(1)}) R^{k+o(1)} = R^{- \alpha \delta + k +o(1)}. \]
Picking $\alpha$ so that $-\delta \alpha + k+k\gamma<0$, we obtain $|B| \le R^{-\delta \alpha + k+k\gamma+o(1)}\to 0 $ as required.

 Suppose instead that the second assumption in Theorem \ref{t:main} holds, so that $K(x)\sim |x|^{-\alpha}$ as $x\to\infty$ for some $\alpha > 0$. Then by the strong splitting in Proposition \ref{p:ssplit} and \eqref{e:pbound}, we have  
 \[ \rho_{r,s}^k\leq R^{o(1)}\bar K(R^{\delta+o(1)}) (\log R)^3 \leq  R^{-\alpha\delta + o(1)} \]
 and so  $|B| \le  R^{-\alpha\delta + k\gamma +o(1)}$.
Choosing $\gamma <\alpha\delta/k$ completes the proof.

\subsubsection{Proof of \eqref{e:sc4}.} 
 Recall that 
\[ 0 \le h^+_R(x_i) \le \eps_R^{-1} \id_{\mathcal{G}_{[x_i, x_i + t_R(a) - \eps_R]}} \]
and define the set
\[ \mathcal{D} = \{x_1,\ldots,x_k \in [-2R,2R] :  [x_i, x_i + t_R(a) - \eps_R] \text{ are pairwise disjoint}\}. \]
In particular, recalling the indicator of distinct gaps $D^{\neq}$, if
\[ \prod_{i=1}^k h^+_R(x_i) D^{\neq}(x_1,\ldots,x_k) > 0 \]
then necessarily $(x_1,\ldots,x_k) \in \mathcal{D}$. Hence it suffices to prove that
 \begin{equation}
     \label{e:sc4reduce}
\eps_R^{-k} \int_{\mathcal{D}} \mathbb P \big[ \cap_{i=1}^k \mathcal{G}_{[x_i, x_i + t_R(a) - \eps_R] }  \big]  (1 - D_R^{\neq}) (x_1,\ldots,x_k)   \, dx_1 \cdots dx_k    \to 0.
 \end{equation}
Fix $x_1,\ldots,x_k \in \mathcal{D}$ such that $D_R^{\neq}(x_1,\ldots,x_k) =0$. Let $(I_i)_{i\ge 1}$ denote the set of disjoint intervals $([x_i, x_i + t_R(a) - \eps_R])_{i\ge 1}$ arranged in increasing order. Let $c > 1$ be a constant which, under the second condition of Theorem \ref{t:main}, we define as in Proposition \ref{p:cssplit} (under the first condition it is arbitrary). Let 
\[ d(I_i, I_{i+1}) = x_{i+1} - (x_i + t_R(a) - \eps_R) \]
be the distance between successive intervals $I_i$ and $I_{i+1}$. Define $r =t_R(a) - \eps_R$ and $s = s_R - r$, which by Claim \ref{c:gap} satisfy $r= R^{o(1)}$ and $s \sim R^\delta$. Then choose $0 \le v \le k$ such that the interval
\[ \big( (2kc)^v s , (2kc)^{v+1} s \big]  \]
does not intersect $(d(I_i, I_{i+1}) )_{1 \le i \le k}$. Consider the graph on $\{1,\ldots,k\}$ with edge ${i,i+1}$ if and only if $d(I_i,I_{i+1}) \le s(2k c)^{v}$. This induces a clustering $(\mathcal{C}_\ell)_{1 \le \ell \le m}$ of the intervals $I_i$. Recall the collection $\bar{\mathcal{Q}}^k(r,\bar r,s)$ from Section \ref{s:split}. We claim that, for sufficiently large $R$, $(\mathcal{C}_\ell)_{1 \le \ell \le m} \in \bar{\mathcal{Q}}^k(r,\bar r, \bar s) $ for $\bar{s} =  s(2kc)^{v+1} $ and  $\bar r = \bar{s} / c $. Indeed each cluster contains at most $k$ intervals which are separated by $\le s(2k c)^{v}$, so is contained in an interval of length at most 
\[ k \big(r + s(2k c)^v \big) = kr +  \frac{ s(2kc)^{v+1}}{2c}  \le  \frac{ s(2kc)^{v+1}}{c}  =: \frac{ \bar s}{c} \]
where the inequality used that $r/s \to 0$ as $R \to \infty$. 

Let $m_s$ denote the number of singleton clusters, and let $m' = m - m_s$ be the number of clusters which contain at least two elements. Without loss of generality we index the latter as $\mathcal{C}_\ell$ for $1 \le \ell \le m'$, and we assume that $m \le k-1$ and $m' \ge 1$, since this is equivalent to $D_R^{\neq}=0$. Define 
\[ P_0 = \prod_{\ell \ge 1} \mathbb P \big[ \cap_{I_i  \in \mathcal{C}_\ell} \mathcal{G}_{I_i} \big] .  \]
Recalling the clustering coefficient $\phi_{r}$ from Section \ref{s:clus}, we have
\[ P_0 \le  \phi_r^{m'} G( r )^{m_s}  \le \phi_r^{m'}  R^{- m_s + o(1) }   . \]
 By Corollary \ref{c:clus}, there exists a $\eta \in (0, 1)$ such that
\[ \phi_r \le G(r)^{1+\eta + o(1)} = R^{-1-\eta + o(1)} , \]
so we have, since we assume $m' \ge 1$,
\begin{equation}
    \label{e:p0bound}
P_0 \le R^{-m' - m'\eta + o(1)} R^{-m_s + o(1)} = R^{-m' - m_s - m' \eta + o(1)} \leq R^{-m - \eta + o(1)}. 
\end{equation}
We next claim that, under the assumptions of Theorem \ref{t:main}, 
\begin{equation}
    \label{e:p0bound2}
 \mathbb P \big[ \cap_{i=1}^k \mathcal{G}_{I_i}  \big] \le  P_0 R^{o(1)}   + R^{-k+o(1)} .
 \end{equation}
Indeed suppose that the first assumption in Theorem \ref{t:main} holds, so in particular $K(x)x^\alpha\to0$ for every $\alpha > 0$. Then by Proposition \ref{p:cwsplit},
\[ \mathbb P \big[ \cap_{i=1}^k \mathcal{G}_{I_i}   \big]   \le P_0 + c_1 r^2 \bar{K}(s) \]
 and the conclusion \eqref{e:p0bound2} follows. Suppose instead that the second assumption in Theorem \ref{t:main} holds, so that $K(x)\sim |x|^{-\alpha}$ as $x\to\infty$ for some $\alpha > 0$. Then by Proposition \ref{p:cssplit},
\[ \mathbb P \big[ \cap_{i=1}^k \mathcal{G}_{I_i}  \big]   \le   c P_0^{1 -  c r^2   \bar{K}(s)  }   = c P_0^{1 +  o(1)}    \]
which again yields \eqref{e:p0bound2}.

To finish we observe also that the subset of $\mathcal{D}$ of corresponding to a given clustering $(\mathcal{C}_\ell)_{1 \le \ell \le m}$ has volume at most 
\[  R^{m} (k  \bar{r} )^k = R^{m + k \delta + o(1) } .\]
Summing up over all possible clusterings, integrating, and combining with \eqref{e:p0bound}--\eqref{e:p0bound2}, the left-hand side of \eqref{e:sc4reduce} is at most
\[ R^{   k \gamma} R^{m + k \delta + o(1) } ( R^{-m -  \eta + o(1)  } + R^{-k + o(1)} ) = R^{   k \gamma + k \delta  -  \eta + o(1) }  \]
where we used that $m \le k-1$ and $\eta \in (0,1)$. Choosing $\delta \in (0,1)$ and $\gamma > 0$ so that $\delta + \gamma < \eta / k$, the previous display tends to zero, completing the proof.

\subsection{A general Poisson approximation result}
\label{ss:gen}
In the proof of Theorem \ref{t:main} we used relatively few properties of the zero set. In view of possible future applications, we collect here the properties that were used:

\begin{proposition}[General Poisson approximation of large gaps]
\label{p:ppp}
Suppose a stationary point process $\mathcal{Z}$ on $\R$ satisfies the following properties, with $G$, $\mathcal{G}_I$, $\rho_{r,s}^k$, and $\phi_r$ defined as above:
\begin{enumerate}
\item (Gap decay) $-\log G(r)  / \log r \to \infty$.
\item (Regularity) $G$ is twice differentiable, convex, and
\[ \log (-G'(r)) \sim \log G(r) \qquad \text{and} \qquad  \log G''(r) \le \log G(r)(1+o(1)) .\]
\item (Splitting) For every $k, \delta > 0$ there exists $\delta' > 0$ such that, for every $r =R^{o(1)}$ and  $s = R^{\delta+o(1)}$,
\[ \rho^k_{r,s} \le R^{-\delta' + o(1) } . \]
\item (Clustered decoupling) For every $k, \delta > 0$ there exists $c > 0$ such that, for every $r =R^{o(1)}$, $s = R^{\delta+o(1)}$, and $\bar{r} \le s /  c$, and every  $(\mathcal{C}_{\ell})_{\ell \ge 1} \in \bar{\mathcal{Q}}^k_{r,\bar r,s}$,
\[  \P \big[ \cap_i \mathcal{G}_{I_i} \big]  \le R^{o(1)} \Big(  \prod_{\ell \ge 1} \P \big[ \cap_{I_i \in \mathcal{C}_\ell}  \mathcal{G}_{I_i} \big]  \Big)^{1 + o(1)} + R^{-k} \Big) .\]
\item (Absence of clustering) There exists $\eta > 0$ such that 
\[ \phi_r \le G(r)^{1+\eta + o(1)} .\]
\end{enumerate}
Then the largest gaps of $\mathcal{Z}$ obey a Poisson approximation in the sense of Definition \ref{d:pa}.
\end{proposition}

\medskip

\bibliographystyle{halpha-abbrv}
\bibliography{gaps}

@article{al21,
    title={Zeros of smooth stationary {G}aussian processes},
    author={M. Ancona and T. Letendre},
    journal={Electron. J. Probab.},
    volume={26},
    number={68},
    year={2021},
    pages={1--81},
}

@article{bb13,
title={Extreme gaps between eigenvalues of random matrices},
author={G. Ben Arous and P. Bourgade},
journal={Ann. Probab.},
volume={41},
number={4},
year={2013},
pages={2648--2681},
	}

@article{bmr20,
	title={A covariance formula for topological events of smooth {G}aussian fields},
	author={D. Beliaev and S. Muirhead and A. Rivera},
	journal={Ann. Probab.},
	volume={86},
	number={6},
	year={2020},
	pages={2845--2893},
	}

@article{bou22,
	title={Extreme gaps between eigenvalues of {W}igner matrices},
	author={P. Bourgade},
	journal={J. Eur. Math. Soc.},
	volume={24},
	number={8},
	pages={2823--2873},
	year={2022},
	}

@article{cc26,
	title={Largest gaps between bulk eigenvalues of unitary-invariant random {H}ermitian matrices},
	author={Christophe Charlier},
	journal={Preprint, arXiv:2602.07524},
    year={2026},
	}

@article{ftw,
	title={Small gaps of {GOE}},
	author={R. Feng and G. Tian and D. Wei},
	journal={Geom. Funct. Anal.},
	volume={29},
	number={6},
	pages={1784--1827},
	year={2019},
	}

@article{fly,
	title={Small gaps of {GSE}},
	author={R. Feng and J. Li and D. Yao},
	journal={Preprint, arXiv:2409.03324},
	year={2024},
	}

@article{MY,
	title={Limit law for root separation in random polynomials},
	author={M. Michelen and O. Yakir},
	journal={Adv. Math.},
    volume={496},
    year={2026},
    pages={110990},
	}

@article{fw,
	title={Small gaps of circular $\beta$-ensemble},
	author={R. Feng and D. Wei},
	journal={Ann. Probab.},
	volume={49},
	number={2},
	pages={997--1032},
	year={2021},
	}

@article{fy,
	title={Smallest distances between zeros of {G}aussian analytic functions},
	author={R. Feng and D. Yao},
	journal={Adv. Math.},
  volume={496},
  year={2026},
  pages={110999},}

@article{fgy24,
	title={Smallest gaps between zeros of stationary {G}aussian processes},
	author={R. Feng and F. G{\"{o}}tze and D. Yao},
	journal={J. Func. Anal.},
	volume={287},
	number={4},
	year={2024},
	pages={110493},
	}

@article{fw25,
	title={Large gaps of {CUE} and {GUE}},
	author={R. Feng and D. Wei},
	journal={Ann. Probab.},
    volume={53},
    number={5},
    pages={1668-1702},
	year={2025},
}

@article{gsw86,
title = {An extreme value theory for long head runs},
author = {L. Gordon and M. F. Schilling and M. S. Waterman},
journal={Probab. Theory Related Fields},
year={1986},
volume={72},
pages={279--287},
}

@article{llm20,
	author={B. Landon and P. Lopatto and J. Marcinek},
	title={Comparison theorem for some extremal eigenvalue statistics},
	journal={Ann. Probab.},
	volume={48},
	number={6},
	year={2020},
	pages={2894--2919},
	}

@book{llr83,
  author = {M.R. Leadbetter and G. Lindgren and H. Rootz\'{e}n},
  title = {Extremes and Related Properties of Random Sequences and Processes},
  publisher = {Springer-Verlag, New York},
  year = {1983},
  }

@article{ffm25,
	author = {N. Feldheim and O. Feldheim and S. Mukherjee},
	title = {Persistence and ball exponents for {G}aussian stationary processes},
	journal = {Commun. Pure Appl. Math.},
    volume={78},
    number  = {10},
    pages={1949–2000},
	year = {2025},
}

@article{ffm26,
	author = {N. Feldheim and O. Feldheim and S. Muirhead},
	title = {Persistence and entropic repulsion of stationary {G}aussian fields with spectral singularity at the origin},
	journal = {Preprint, arXiv:2605.20587},
	year = {2026},
}

@article{mui23,
	author={S. Muirhead},
	title={A sprinkled decoupling inequality for {G}aussian processes and applications},
	journal={Electron. J. Probab.},
	volume={28},
	pages={1--25},
	year={2023},
}

@article{mm25,
	author={M. McAuley and S. Muirhead},
	title={Limit theorems for the number of sign and level-set clusters of the {G}aussian free field},
    journal={Preprint, arXiv:2501.14707},
	year={2025},
	}

@article{lo25,
	author={P. Lopatto and M. Otto},
	title={Maximum gap in complex {G}inibre matrices},
    journal={Preprint, arXiv:2501.04611},
	year={2025},
	}

@book{at07,
	Author = {R. Adler and J. Taylor},
	Publisher = {Springer},
	Title = {Random fields and geometry},
	Year = {2007}}

@book{kall,
author={O. Kallenberg},
title={Random Measures},
publisher={Akademie-Verlag Berlin},
year={1986},
}

@article{nov88,
author={S. Y. Novak},
title={Time intervals of constant sojourn of a homogeneous {M}arkov chain in a fixed subset of states},
journal={Siberian Math. J.},
volume={29},
number={1},
year={1988},
pages={100--109}
}

@article{nov92,
title={Longest runs in a sequence of m-dependent random variables},
author={S. Y. Novak},
journal={Probab. Theory Related Fields},
volume={91},
pages={269--281},
year={1992}
}

@Book{piterbarg,
    Author = {V.I. Piterbarg},
    Title = {Asymptotic Methods in the Theory of Gaussian Processes and Fields},
    ISBN = {0-8218-0423-5},
    Pages = {xii + 206},
    Year = {1996},
    Publisher = {Providence, RI: AMS},
    Language = {English},
    MSC2010 = {60G15 60-02 60G60 60G17},
    Zbl = {0841.60024}
}

@article{c25,
	author={C. Charlier},
	title={Smallest gaps of the two-dimensional {C}oulomb gas},
    journal={Preprint, arXiv:2507.23502},
	year={2025},
	}

\end{document}